\documentclass[12pt]{article}
\usepackage[utf8]{inputenc}
\usepackage{amsfonts}
\usepackage{amsmath}
\usepackage{amssymb}
\usepackage{amsthm}
\usepackage{enumerate}
\usepackage{enumitem}
\usepackage{epsfig}
\usepackage{color}
\usepackage{mathtools}
\usepackage{wrapfig}

\usepackage{marginnote}

\usepackage{float}

\usepackage[authoryear,sort,round]{natbib}  

\usepackage{subfigure}

\usepackage[normalem]{ulem}
\usepackage{soul}
\usepackage[dvipsnames]{xcolor}
\usepackage{todonotes}

\usepackage{thmtools}
\usepackage{thm-restate}

\usepackage{mathptmx} 

\usepackage[backref=page]{hyperref}

 \renewcommand*{\backrefalt}[4]{%
    \ifcase #1%
     \or (see page:~#2)%
     \else (see pages:~#2)%
    \fi%
    }

\usepackage{nicefrac,cancel}
\usepackage[noabbrev,capitalise,nosort,nameinlink]{cleveref}

\crefname{openproblem}{Open Problem}{Open Problem}
\crefname{notation}{Notation}{Notation}

\newcommand*{\arXiv}[1]{\bgroup\color{blue}\href{https://arxiv.org/abs/#1}{arXiv:#1}\egroup}
\newcommand*{\doi}[1]{\bgroup\color{blue}\href{https://doi.org/#1}{doi:#1}\egroup}
\newcommand*{\email}[1]{\bgroup\color{blue}\href{mailto:#1}{#1}\egroup}
\renewcommand*{\url}[1]{\bgroup\color{blue}\href{#1}{#1}\egroup}

\renewcommand{\d}{\mathrm d}
\newcommand{\R}{\mathbb R}
\newcommand{\N}{\mathbb N}
\newcommand{\calN}{\mathcal N}
\newcommand{\cH}{\mathcal H}

\newcommand{\eps}{\varepsilon}

\newcommand{\diag}{\mathrm{diag}}
\DeclareMathOperator{\argmax}{argmax}

\makeatletter
\newcommand*\quark{\mathpalette\quark@{.5}}
\newcommand*\quark@[2]{\mathbin{\vcenter{\hbox{\scalebox{#2}{$\; \m@th#1\bullet \;$}}}}}
\makeatother

\usepackage[a4paper, lmargin=0.1666\paperwidth, rmargin=0.1666\paperwidth, tmargin=0.1111\paperheight, bmargin=0.1111\paperheight]{geometry} 
\usepackage{parskip}
\usepackage[all]{nowidow} 
\usepackage[protrusion=true,expansion=true]{microtype} 
\usepackage{breakcites}

\usepackage{lipsum} 

\theoremstyle{plain}
\newtheorem{theoremX}{\sffamily Theorem}[section]
\newenvironment{theorem}
  {\pushQED{\qed}\theoremX}
  {\popQED\endtheoremX}
\crefname{theoremX}{theorem}{theorems}  
\Crefname{theoremX}{Theorem}{Theorems}  
  
\newtheorem{proposition}[theoremX]{\sffamily Proposition}
\newtheorem{lemmaX}[theoremX]{\sffamily Lemma}
\newenvironment{lemma}
  {\pushQED{\qed}\lemmaX}
  {\popQED\endlemmaX}
\crefname{lemmaX}{lemma}{lemmata}  
\Crefname{lemmaX}{Lemma}{Lemmata}

\newtheorem{corollaryX}[theoremX]{\sffamily Corollary}
\newenvironment{corollary}
  {\pushQED{\qed}\corollaryX}
  {\popQED\endcorollaryX}
\crefname{corollaryX}{Corollary}{Corollaries}
\Crefname{corollaryX}{Corollary}{Corollaries}

\newtheorem{conjectureX}[theoremX]{\sffamily Conjecture}
\newenvironment{conjecture}
  {\pushQED{\qed}\conjectureX}
  {\popQED\endconjectureX}
\crefname{conjectureX}{Conjecture}{Conjectures}
\Crefname{conjectureX}{Conjecture}{Conjectures}

\theoremstyle{definition}
\newtheorem{definitionX}[theoremX]{\sffamily Definition}
\newenvironment{definition}
  {\pushQED{\qed}\definitionX}
  {\popQED\enddefinitionX}
\crefname{definitionX}{definition}{definitions}  
\Crefname{definitionX}{Definition}{Definitions}  
\crefname{condition}{Conditions}{Conditions} 
\Crefname{condition}{Condition}{Conditions}

\newtheorem{notation}[theoremX]{\sffamily Notation}

\newtheorem{condition}[theoremX]{\sffamily Condition}
\newtheorem{remarkX}[theoremX]{\sffamily Remark}
\newenvironment{remark}
  {\pushQED{\qed}\remarkX}
  {\popQED\endremarkX}

\newtheorem{assumptionX}[theoremX]{\sffamily Assumption}
\newenvironment{assumption}
  {\pushQED{\qed}\assumptionX}
  {\popQED\endassumptionX}
\crefname{assumptionX}{assumption}{assumptions}  
\Crefname{assumptionX}{Assumption}{Assumptions}

\newcommand{\absval}[1]{\lvert #1 \rvert}
\newcommand{\innerprod}[2]{\langle #1 , #2 \rangle}
\newcommand{\norm}[1]{\lVert #1 \rVert}

\newcommand{\Absval}[1]{\left\vert #1 \right\vert}

\newcommand{\rat}{\,\mathfrak R}
\newcommand{\ratio}[3]{\tooltip****[black]{\ensuremath{\rat(#1,#2,#3)}}{$\frac{\mu(B_{#3}(#1))  }{\mu(B_{#3}(#2))}$}}
\newcommand{\ratioy}[3]{\tooltip****[black]{\ensuremath{\rat^y(#1,#2,#3)}}{$\frac{\mu^y(B_{#3}(#1))  }{\mu^y(B_{#3}(#2))}$}}

\newcommand{\Ratio}[4]{\ensuremath{\rat_{#4}^{#3}(#1,#2)}}
\renewcommand{\ratio}[3]{\ensuremath{\rat_{\mu}^{#3}(#1,#2)}}
\renewcommand{\ratioy}[3]{\ensuremath{\rat_{\mu^{y}}^{#3}(#1,#2)}}
\numberwithin{equation}{section}
\numberwithin{figure}{section}
\numberwithin{table}{section}

\usepackage{adjustbox}
\usepackage{tikz}
\usetikzlibrary{bayesnet,er,trees,shapes.symbols,mindmap,arrows,decorations,shapes.misc,shapes.arrows,chains,matrix,positioning,scopes,decorations.pathmorphing,patterns}
\usepackage{tikz-cd} 
\newcommand\mlnode[1]{\fbox{\begin{tabular}{@{}c@{}}#1\end{tabular}}} 

\definecolor{darkgreen}{rgb}{0,0.5,0}

\hypersetup{ 	
pdfsubject = {},
pdftitle = {},
pdfauthor = {}
}

\author{Ilja Klebanov \and Philipp Wacker}
\title{Maximum a posteriori estimators in $\ell^p$ are well-defined for diagonal Gaussian priors.}

\begin{document} 
\maketitle

\begin{abstract}
We prove that maximum a posteriori estimators are well-defined for diagonal Gaussian priors $\mu$ on $\ell^p$ under common assumptions on the potential $\Phi$.
Further, we show connections to the Onsager--Machlup functional and provide a corrected and strongly simplified proof in the Hilbert space case $p=2$, previously established by \citet{dashti2013map,kretschmann2019nonparametric}.

These corrections do not generalize to the setting $1 \leq p < \infty$, which requires a novel convexification result for the difference between the Cameron--Martin norm and the $p$-norm.

\end{abstract}


\tableofcontents

\section{Introduction}
\label{section:Introduction}
Let $(X,\|\quark \|_{X})$, be a separable Banach space and $\mu$ a centred and non-degenerate Gaussian (prior) probability measure on $X$. We are motivated by the inverse problem of inferring the unknown parameter $u\in X$ via noisy measurements
\begin{equation}
y = G(u) + \epsilon \label{eq:invProb},
\end{equation}
where $G: X\to \mathbb R^d$ is a (possibly nonlinear) measurement operator and $\epsilon$ is measurement noise, typically assumed to be independent of $u$.
The Bayesian approach to solving such inverse problems \citep{stuart2010inverse} is to combine prior knowledge given by $\mu$ with the data-dependent likelihood into the posterior distribution $\mu^{y}$ given by
\begin{equation}
\label{eq:post_pr}
\frac{\d\mu^y}{\d\mu}(u) = Z^{-1}\cdot \exp(-\Phi(u)).
\end{equation}
Here, the so-called potential $\Phi: X\to \R$ depends on $G$ and the statistical structure of the measurement noise $\eps$, while $Z\coloneqq \int_{X} \exp(-\Phi(u)) \mu(\mathrm du)$ is simply the normalization constant, which is well defined under suitable conditions on $\Phi$ (see \Cref{ass:Phi} later on). If, for example, the measurement noise is distributed according to a centred Gaussian measure on $\mathbb R^d$, $\varepsilon \sim N(0, \Gamma)$ with symmetric and positive definite covariance matrix $\Gamma \in \mathbb R^{d\times d}$, then $\Phi(u) = \frac1 2 \|\Gamma^{-1/2}(y-G(u))\|^2$, but we will use general formulation (\Cref{eq:post_pr}) as the starting point for our considerations. For an overview of the Bayesian approach to inverse problems and a discussion of its well-posedness we refer to \citep{stuart2010inverse} and the references therein.

Our focus lies on the analysis of the so-called ``maximum a posteriori (MAP) estimator'' or ``mode'', i.e.\ the summary of the posterior $\mu^{y}$ in the form of a single point $u_{\textup{MAP}} \in X$.
In the finite-dimensional setting $X = \R^{k}$, if $\mu^{y}$ has a continuous Lebesgue density $\rho^{y}$, MAP estimators are simply defined as the parameter of highest posterior density, $u_{\textup{MAP}} = \argmax_{u\in\R^{k}} \rho^{y}(u)$ (note that such maximizers may not be unique or fail to exist).

Unfortunately, this definition does not generalize to measures without a continuous Lebesgue density, in particular it can not cover infinite-dimensional settings, where there is no equivalent of the Lebesgue measure.

For this reason \citet[Definition 3.1]{dashti2013map} suggested to define MAP estimators as ``maximizers of infinitesimally small ball (posterior) mass'', see \Cref{def:MAP} below.
To simplify notation, we first introduce the following shorthand for the ratios of ball masses:

\begin{notation}
For a separable metric space $X$ and a probability measure $\nu$ on $X$, we denote the open ball of radius $\delta > 0$ centred at $x\in X$ by $B_{\delta}(x)$.
Further, for $w,z \in X$ with $\nu(B_\delta(z)) > 0$, we set
\[
\Ratio{w}{z}{\delta}{\nu}
\coloneqq
\frac{\nu(B_\delta(w))}{\nu(B_\delta(z))},
\qquad
\Ratio{w}{\sup}{\delta}{\nu}
\coloneqq
\frac{\nu(B_\delta(w))}{\sup_{z\in X}\nu(B_\delta(z))}.
\]
Similarly, we set $\Ratio{\sup}{w}{\delta}{\nu} \coloneqq \Ratio{w}{\sup}{\delta}{\nu}^{-1}$ whenever $\nu(B_\delta(w)) \neq 0$.
\end{notation}

\begin{remark}
Note that $\sup_{z\in X}\nu(B_\delta(z)) > 0$ follows from the separability of $X$:
Assume that $(z_{n})_{n\in\N}$ is dense in $X$, $\delta > 0$ and $\nu(B_{\delta}(z_{n})) = 0$ for each $n\in\N$.
Then $\nu(X) \leq \sum_{n\in\N} \nu(B_{\delta}(z_{n})) = 0$ (since $X \subseteq \bigcup_{n\in\N} B_{\delta}(z_{n})$) and $\nu$ could not be a probability measure.
\end{remark}

We work with the following rather general definition of MAP estimators:

\begin{definition}[{\citealt[Definition 3.6]{ayanbayev2021gamma}}]
\label{def:MAP}
Let $X$ be a separable metric space and $\nu$ be a probability measure on $X$.
A \emph{strong mode} for $\nu$ is any $z\in X$ satisfying
\begin{equation}\label{eq:defMAP}
\lim_{\delta\searrow 0} \Ratio{z}{\sup}{\delta}{\nu} = 1.
\end{equation}
If $\nu = \mu^y$ is a Bayesian posterior measure given by \eqref{eq:post_pr}, then we call any strong mode a \emph{MAP estimator}.\qedhere
\end{definition}

Other sources, especially from the physics community, see e.g. \citep{durr1978onsager}, (informally) define the MAP estimator as the minimizer of the so-called Onsager--Machlup (OM) functional, which can be thought of as a generalization of the negative posterior log-density \citep{dashti2013map}:
\begin{definition}
Let $\mu$ be a Gaussian (prior) measure on a separable Banach space $X$ with Cameron--Martin space $(E,\absval{\quark}_{E})$ and $\Phi\colon X \to \R$ be such that $\exp(-\Phi)$ is $\mu$-integrable.
We define the \textit{Onsager-Machlup (OM) functional} $I\colon E\to \R$ corresponding to $\mu^{y}$ given by \eqref{eq:post_pr} by
\begin{equation}
    \label{eq:OM} I(u) := \Phi(u) + \frac{1}{2}|u|_E^2.\qedhere
\end{equation}
\end{definition}

The connection between between OM minimizers and MAP estimators is non-trivial in general separable Banach spaces.\footnote{Note that \citep[Theorem 3.2]{dashti2013map}, restated as \Cref{thm:OM} below, only gives partial answers, since only pairwise comparisons of points lying in $E$ are made,
while \citep[Proposition 4.1]{ayanbayev2021gamma} makes the connection between OM minimizers and weak modes (rather than strong modes, which correspond to MAP estimators) under different assumptions.}
Natural questions arising in this context are
\begin{itemize}
    \item whether (or under which conditions) MAP estimators exist and
    \item whether MAP estimators can equivalently be characterized as minimizers of the OM functional.
\end{itemize}

One fundamental ingredient, and the most direct reason why small-ball probabilities are related to the functional $I$, is the following theorem about the Onsager-Machlup functional:
\begin{theorem}[{\citealp[Theorem 3.2]{dashti2013map}}]
\label{thm:OM}
Let \Cref{ass:Phi} hold. Then for $z_1,z_2\in E$,
\[
\lim_{\delta\searrow 0} \ratioy{z_1}{z_2}{\delta}= \exp(I(z_2)-I(z_1)).\qedhere
\]
\end{theorem}

However, \Cref{thm:OM} does not yield the full answer regarding the connection of MAP estimators and OM minimizers --- not only is it restricted to elements of the Cameron--Martin space $E$, also it only provides \emph{pairwise} comparisons of two points $z_{1},z_{2}\in E$, while MAP estimators require consideration of the ratio $\ratioy{z_1}{\sup}{\delta}$ and its limit as $\delta\searrow 0$.

\begin{remark}
Note that $I$ amounts to a Tikhonov-Phillips regularization of the misfit functional $\Phi$, so the results in this manuscript are also to be understood in
the context of regularized optimization.
\end{remark}

\citet{dashti2013map} discussed, for the first time, the existence of MAP estimators as well as their connection to minimizers of the OM functional, in the specific setting of a Bayesian inverse problem of type \eqref{eq:invProb}.
More precisely, they claim to prove the following statements for every separable Banach space $X$ under \Cref{ass:Phi} below \citep[Theorem 3.5]{dashti2013map}:
\begin{enumerate}[label=(\Roman*)]
    \item \label{item:first_statement_dashti} Let $z^{\delta} = \argmax_{z\in X} \mu^{y}(B_{\delta}(z))$.
    There exists a subsequence of $(z^{\delta})_{\delta > 0}$ that converges strongly in $X$ to some element $\overline{z}\in E$.
    \item \label{item:second_statement_dashti} The limit $\overline{z}$ is a MAP estimator of $\mu^{y}$ (this proves existence of such an object) and it is a minimizer of the OM functional.
\end{enumerate}

However, while the ideas of \citet{dashti2013map} are groundbreaking, their proof of the above statements, as well as the corrections provided by \citet{kretschmann2019nonparametric}, rely on techniques that hold in separable Hilbert spaces rather than separable Banach spaces, see \Cref{section:Mistakes_Dashti}.

Further, neither \citet{dashti2013map} nor \citet{kretschmann2019nonparametric} show the existence of the $\delta$-ball maximizers $z^{\delta}$ above, which are the central objects in their proofs.
It turns out that the existence of $z^{\delta}$ is a highly non-trivial issue and has recently been discussed by \citet{hefin}, who proved their existence for certain measures (including posteriors arising from non-degenerate Gaussian priors on $\ell^p$) and gave counterexamples for others.

Our approach relies on asymptotic maximizers in the following sense, which are guaranteed to exist by the definition of the supremum (in fact, even for arbitrary families $(\eps^\delta)_{\delta > 0}$ in $(0,1)$).

\begin{definition}\label{def:asymptoticmaximizer}
Let $X$ be a separable metric space and $\nu$ be a probability measure on $X$.
A family $(\zeta^\delta)_{\delta > 0}\subset X$ is called an \textit{asymptotic maximizing family (AMF)} for $\nu$, if there exists a family $(\eps^\delta)_{\delta > 0}$ in $(0,1)$ such that $\eps^\delta \searrow 0$ as $\delta\searrow 0$ and, for each $\delta > 0$,
\begin{equation}
\label{eq:asymptoticmaximizer}
    \Ratio{\zeta^\delta}{\sup}{\delta}{\nu} > 1-\eps^\delta.\qedhere
\end{equation}
\end{definition}

\begin{lemma}
\label{lem:MAP_AMF}
For any separable metric space $X$ and any probability measure $\nu$ on $X$, there exists an AMF for $\nu$.
Further, if $\bar{z}$ is a MAP estimator for $\nu$, then the constant family $(\bar{z})_{\delta > 0}$ forms an AMF for $\nu$.
\end{lemma}

\begin{proof}
This follows directly from the definition of the supremum (in fact, for \emph{any} family $(\eps^\delta)_{\delta > 0}$ a corresponding AMF can be found) and \Cref{def:MAP,def:asymptoticmaximizer}.
\end{proof}

The corresponding statements to \ref{item:first_statement_dashti}--\ref{item:second_statement_dashti} are given in \Cref{conj:main_statement}.
Note that we strengthened those statements by stating the equivalence of MAP estimators, minimizers of the OM functional and limit points of AMFs.
Especially the latter can not be expected for the $\delta$-ball maximizers $z^{\delta}$, even when they exist and are unique, since it is easy to construct MAP estimators that are not limit points of $(z^{\delta})_{\delta > 0}$ as $\delta \searrow 0$, even for continuous measures on $\R^{1}$.
Apart from their guaranteed existence, this is yet another advantage of working with AMFs $(\zeta^{\delta})_{\delta > 0}$ rather than with $(z^{\delta})_{\delta > 0}$.

\subsection{Why this paper is necessary}
\label{section:Mistakes_Dashti}

The contribution of this paper is twofold:
\begin{enumerate}
\item 
remedy the crucial shortcomings of previous work on the existence of MAP estimators mentioned above and listed in detail below, resulting in a corrected and strongly simplified proof of the existence of MAP estimators in the Hilbert space setting (\Cref{thm:main_hilbert}, proven in \Cref{section:Hilbert});
\item
generalize the corresponding result from Hilbert spaces to sequence spaces $X = \ell^{p}$, $1 \leq p < \infty$, of $p^{\textup{th}}$-power summable sequences and diagonal\footnote{By ``diagonal'' we mean that $\mu=\otimes_{k\in\N} \mathcal N(0,\sigma_k^2)$ has a diagonal covariance structure with respect to the canonical basis, while ``nondegenerate'' refers to the fact that the eigenvalues of the covariance operator are strictly positive, $\sigma_{k}^{2} > 0$ for $k\in\N$.
Note that Gaussian measures on separable Hilbert spaces can always be diagonalized in this sense by choosing an orthonormal eigenbasis of the covariance operator, see \Cref{notation:Hilbert_setting}, hence our results constitute a genuine generalization of the Hilbert space case.}
and nondegenerate Gaussian prior measures,  proven in \Cref{section:proof_lp}).
For this purpose, we develop a novel and non-trivial convexification argument for the difference between the Cameron--Martin norm $|\quark|_E$ and the ambient space norm $\|\quark\|_X$ in \Cref{prop:convexitybound_lp_all_p}.
\end{enumerate}

The shortcomings of previous work on the existence of MAP estimators include:
\begin{itemize}
    \item The crucial object in the proofs of \citep{dashti2013map}, \[z^{\delta} = \argmax_{z\in X} \mu^{y}(B_{\delta}(z)),\] is defined without a proof of its existence.
    This is a highly non-trivial issue which was not fixed by the corrections in  \citet{kretschmann2019nonparametric}.
    In \citep[Example 4.8]{hefin}, the authors construct a probability measure on a separable metric space without such $\delta$-ball maximizers $z^\delta$, but prove in \citep[Corollary 4.10]{hefin} that such maximizers exist for posteriors arising from non-degenerate Gaussian priors on $\ell^p$.    
    \item Specific Hilbert space properties are used in Banach spaces, in particular, the proof of \citep[Theorem 3.5]{dashti2013map} relies heavily on the existence of an orthogonal basis of the Cameron--Martin space which satisfies $\norm{x}_{X}^{2} = \sum_{n\in\N} x_{n}^{2}$ for $x\in X$, where $x_{n}$ are the coordinates of $x$ in that basis.
    \item While the defining property of a MAP estimator $z\in X$ is given by \[\lim_{\delta\searrow 0} \ratioy{z}{\sup}{\delta} = 1,\] the proof of \citep[Theorem 3.5]{dashti2013map} considers this limit only for a specific null sequence $(\delta_{m})_{m\in\N}$.
    This is hidden in their notation, where, for simplicity, they adopt the notation $(z^{\delta})_{\delta > 0}$ for subsequences --- a rather typical abuse of notation which is illegitimate in this specific case, since different null sequences $(\delta_{m})_{m \in \N}$ can yield different candidates for MAP estimators.
    \item While \citet[Lemma 3.9]{dashti2013map} is stated for $\bar{z} = 0$, it is later applied to more general $\bar{z}\in X$.
    In Banach spaces, validity of this substitution is equivalent to tacit assumption of the Radon--Riesz property, which only holds for a strict subset of separable Banach spaces (and excludes the paradigmatic case $X = \ell^1$).
    \item The proof of \citep[Corollary 3.10]{dashti2013map} relies on MAP estimators being limit points of $(z^{\delta})_{\delta > 0}$.
    However, only the reverse implication had been discussed, and, in fact, this implication is incorrect even when $z^{\delta}$, $\delta > 0$, is guaranteed to exist, as can be easily seen from the following simple example of a bimodal distribution on $\R^{1}$:
    Let $0 < \sigma < 1$ and $\mu^y$ have Lebesgue density
    $\rho^y(x) \propto \exp(-(x-1)^2/2)\chi_{\R^+} + \exp(-(x+1)^2/(2\cdot \sigma^2))\chi_{\R^-}$.
    Then $z^\delta = 1$ for all $\delta < \frac{1}{2}$, but both $x=\pm 1$ are true MAP estimators.
    For this purpose, we work with AMFs introduced in \Cref{def:asymptoticmaximizer}, the limit points of which we show to coincide with MAP estimators.
\end{itemize}

\Cref{conj:main_statement} in general separable Banach spaces and general Gaussian measures remains unsolved and is an extremely intricate issue.
The ``skeleton'' of our proofs is provided in \Cref{thm:main_general}, where the main steps are shown under suitable conditions (while proving those conditions in specific settings typically requires a lot of work).
This establishes a framework for proving \Cref{conj:main_statement} in other Banach spaces, thereby paving the road for future research on this topic.

\subsection{Related Work}
\label{sec:relwork}

The definition of strong modes by \cite{dashti2013map} has sparked a series of papers with variations on this concept, most notably generalized strong modes \cite{clason2019generalized}, weak modes \citep{helin2015maximum}. \citep{agapiou2018sparsity} studied the MAP estimator for Bayesian inversion with sparsity-promoting Besov priors. The connection between weak and strong modes was further explored in \cite{lie2018equivalence}, and \cite{ayanbayev2021gamma,ayanbayev2021gammab} discussed stability and convergence of global weak modes using $\Gamma$-convergence. Recently, \cite{hefin} presented a perspective on modes via order theory.

\subsection{Structure of this manuscript}
\Cref{section:Main_results} describes the common framework along which the well-definedness of MAP estimators can be proven in all cases considered (Hilbert space and $X=\ell^p$) and, possibly, further separable Banach spaces. 
\Cref{section:Hilbert} and \Cref{section:proof_lp} apply this framework in order to prove well-definedness of the MAP estimator in the Hilbert space  and $\ell^p$ case, respectively. 

\section{Existence of maximum-a-posteriori estimators}
\label{section:Main_results}
This section covers all the main results mentioned in the introduction. Throughout the paper, we will make the following general assumptions:
\begin{assumption}\label{ass:Phi}
Let $(X,\|\quark\|_X)$ be a separable Banach space, which we call the ambient space, and $\mu$ be a non-degenerate centred Gaussian (prior) probability measure on $X$.
Let $(E,\absval{\quark}_{E})$ denote the corresponding Cameron--Martin space and $\mu^{y}$ be the (posterior) probability measure on $X$ given by \eqref{eq:post_pr}, where the potential $\Phi \colon X \to \R$ satisfies the following conditions:
\begin{enumerate}[label=(\alph*)]
    \item\label{item:first_condition_Phi} $\Phi$ is globally bounded from below, i.e.\ there exists $M\in \R$ such that for all $u\in X$,
    \[\Phi(u)\geq M.\]
    \item\label{item:second_condition_Phi} $\Phi$ is locally bounded from above, i.e.\ for every $r>0$ there exists $K(r) > 0$ such that for all $u\in X$ with $\|u\|_X < r$ we have 
    \[\Phi(u) \leq K(r).\]
    \item\label{item:third_condition_Phi} $\Phi$ is locally Lipschitz continuous, i.e.\ for every $r > 0$ there exists $L(r) > 0$ such that for all $u_1,u_2\in X$ with $\|u_1\|_X,\|u_2\|_X \leq r$ we have
    \[|\Phi(u_1)-\Phi(u_2)|\leq L(r)\, \|u_1-u_2\|_X.\]
\end{enumerate}
Purely for convenience, we assume that $\Phi(0) = 0$. This can be easily achieved by subtracting $\Phi(0)$ from $\Phi$ and incorporating the resulting additional prefactor into the normalization constant $Z$ in \eqref{eq:post_pr}.
\end{assumption}

\begin{remark}
Conditions \ref{item:first_condition_Phi}--\ref{item:third_condition_Phi} are identical to \citep[Assumption 2.1]{dashti2013map}, except that \ref{item:first_condition_Phi} is slightly stronger: \citep{dashti2013map} initially assume the weaker inequality $\Phi(u)\geq M - \varepsilon \|u\|_{X}^{2}$ for every $\varepsilon > 0$, but also make the additional assumption of global boundedness from below (in the sense of (a)  in \Cref{ass:Phi}) in their main theorem 3.5. This assumption is usually not too restrictive as our condition \ref{item:first_condition_Phi} still covers most practical Bayesian inverse problems, since $\Phi$ is typically even non-negative (cf. introduction). 
Further, the non-degeneracy of $\mu$ together with the above conditions guarantees that the ratios $\Ratio{w}{z}{\delta}{\mu}$ and $\Ratio{w}{z}{\delta}{\mu^{y}}$ etc.\ are always well-defined.
Given the assumption $\Phi(0) = 0$, condition \ref{item:second_condition_Phi} is an implication of \ref{item:third_condition_Phi}, but we keep the conditions separated for didactical reasons and comparability to previous papers.
\end{remark}

First, let us restate the result in \citep[Theorem 3.5]{dashti2013map} as a conjecture, since their proof is only (partially, due to unclear existence of $\delta$-ball maximizing centers $z^\delta$) correct in Hilbert spaces and the Banach space version remains an open problem:

\begin{conjecture}\label{conj:main_statement}
Let \Cref{ass:Phi} hold. Then: 
\begin{enumerate}[label=(\alph*)]
\item The following statements are equivalent:
\begin{enumerate}[label=(\roman*)]
    \item $\bar z$ is an $X$-strong limit point as $\delta\to 0$ of some asymptotic maximizing family (AMF) for $\mu^{y}$.\footnote{I.e., there exists a sequence $(\delta_n)_{n\in \N}$ with $\delta_n\searrow 0$ such that $\|\zeta^{\delta_n}-\bar z\|_X \to 0$ as $n\to \infty$.}
    \item $\bar z\in E$ and $\bar z$ minimizes the OM functional.
    \item $\bar z$ is a MAP estimator.
\end{enumerate}
\item There exists at least one MAP estimator.
\qedhere
\end{enumerate}
\end{conjecture}

The main goal of this paper is to provide proofs of \Cref{conj:main_statement} in the special cases where
\begin{itemize}
\item
$X$ is a separable Hilbert space (\Cref{thm:main_hilbert}), where we correct and strongly simplify the proofs initially proposed by \citep{dashti2013map} and worked out in detail in the PhD thesis of \citet{kretschmann2019nonparametric}, or
\item
$X = \ell^{p}$ with $p\in [1,\infty)$ and $\mu=\otimes_{k\in\N} \mathcal N(0,\sigma_k^2)$ is a diagonal Gaussian measure on $X$ (\Cref{thm:main_lp}), which is an entirely new result.
\end{itemize}

\begin{restatable}{theorem}{mainthmhilbert}
\label{thm:main_hilbert}
Let \Cref{ass:Phi} hold. Then \Cref{conj:main_statement} holds for any separable Hilbert space $X=\cH$.
\end{restatable}
\begin{proof}
See \Cref{section:Hilbert}.
\end{proof}
\begin{restatable}{theorem}{mainthmabstract}
\label{thm:main_lp}
Let \Cref{ass:Phi} hold. Then \Cref{conj:main_statement} holds for $X = \ell^p$, $p\in[1,\infty)$, and any diagonal Gaussian (prior) measure $\mu=\otimes_k \mathcal N(0,\sigma_k^2)$ on $X$.
\end{restatable}
\begin{proof}
See \Cref{section:proof_lp}.
\end{proof}

\subsection{Proof strategy}\label{sec:proofstrategy}

In order to prove \Cref{thm:main_hilbert,thm:main_lp}, we proceed along the following seven steps, where $(\zeta^{\delta})_{\delta > 0}$ is an arbitrary AMF for $\mu^{y}$ and $(\delta_{m})_{m\in\N}$ denotes an arbitrary null sequence.
This is a rather general approach and can be followed to prove \Cref{conj:main_statement} for further classes of Banach spaces.

\begin{enumerate}[label = (\roman*)]
\item
\label{item:proof_strategy_boundedness}
Show that $(\zeta^{\delta_{m}})_{m\in\N}$ is bounded.
\item 
\label{item:proof_strategy_subsequence}
Extract a weakly convergent subsequence of $(\zeta^{\delta_{m}})_{m\in\N}$, which, for simplicity, we denote by the same symbol, with weak limit $\bar{z} \in X$.
\item 
\label{item:proof_strategy_limit_in_E}
Prove that $\bar{z}$ lies in the Cameron--Martin space $E$.
\item
\label{item:proof_strategy_convergence_strong}
Show that the convergence is, in fact, strong: $\norm{\zeta^{\delta_{m}} - \bar{z}}_{X} \to 0$ as $m\to\infty$.
\item
\label{item:proof_strategy_MAP}
Infer that any limit point $\bar{z}$ of an AMF (not just the one obtained in \ref{item:proof_strategy_subsequence}--\ref{item:proof_strategy_convergence_strong})  is a MAP estimator, proving its existence.
\item
\label{item:proof_strategy_OM_minimizer}
Prove that any MAP estimator minimizes the OM functional and is a limit point of some AMF.
\item
\label{item:proof_strategy_OM_MAP}
Show that any OM minimizer is also a MAP estimator.
\end{enumerate}

An illustration how this proof strategy fits within the context of \Cref{conj:main_statement} can be found in \Cref{fig:AMF_MAP_OM}.

\begin{figure}[H]
	\centering
		\begin{tikzcd}[column sep=3em,row sep=-0.3em]	
			\mlnode{limit point of some AMF}
			\arrow[r,Rightarrow,shift left = 0.5em,"\text{\ref{item:proof_strategy_MAP}}"]
			&
			\mlnode{MAP estimator}
			\arrow[r,Rightarrow,shift left = 0.5em,"\text{\ref{item:proof_strategy_OM_minimizer}}"]
			\arrow[l,Rightarrow,shift left = 0.5em,"\text{\ref{item:proof_strategy_OM_minimizer}}"]
			&
			\mlnode{OM minimizer}
			\arrow[l,Rightarrow,shift left = 0.5em,"\text{\ref{item:proof_strategy_OM_MAP}}"]
			\\
			\text{\textcolor{gray}{existence via \ref{item:proof_strategy_boundedness} -- \ref{item:proof_strategy_convergence_strong}}}
			& &
		\end{tikzcd}
	\caption{Strategy for proving the existence and equivalence of AMF limit points, MAP estimators and OM minimizers.}
	\label{fig:AMF_MAP_OM}
\end{figure}
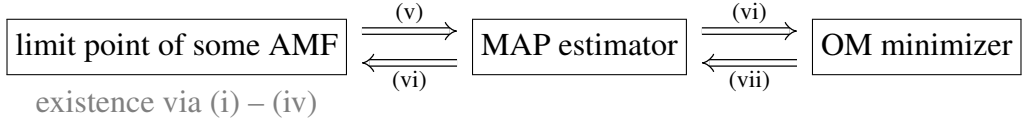

The proof of \ref{item:proof_strategy_boundedness}, \ref{item:proof_strategy_limit_in_E} and \ref{item:proof_strategy_convergence_strong} 
is highly non-trivial and relies on the following idea:
First, we prove that, under \Cref{ass:Phi}, the fraction $\Ratio{\zeta^{\delta_m}}{0}{\delta_{m}}{\mu}$ is bounded away from $0$, meaning that the $\zeta^{\delta_m}$ do not carry negligible prior ball mass in the asymptotic limit.
Second, we show for any sequence $(x_{m})_{m\in\N}$ in $X$ that, if either
\begin{itemize}
\item $(x_{m})_{m\in\N}$ is unbounded or
\item $x_{m} \rightharpoonup \bar{z}$ with $\bar{z}\notin E$ or
\item $x_{m} \rightharpoonup \bar{z} \in E$ but $\norm{x_{m} - \bar{z}}_{X} \not\to 0$,
\end{itemize}
then
\begin{equation*}
\liminf_{m \to \infty}\, \Ratio{x_m}{0}{\delta_{m}}{\mu} 
=
0,
\end{equation*}
providing a contradiction for $x_{m} = \zeta^{\delta_{m}}$.
The three properties described above, as well as \ref{item:proof_strategy_subsequence}, are formulated in \Cref{cond:main_conditions} \ref{item:main_conditions_general_decay}---\ref{item:main_conditions_weak_not_strong} and stated as assumptions in \Cref{thm:main_general}, which can therefore be seen as a ``shell theorem''.
Note that steps \ref{item:proof_strategy_MAP}, \ref{item:proof_strategy_OM_minimizer} and \ref{item:proof_strategy_OM_MAP} then follow in any separable Banach space.

Finally, we prove \Cref{cond:main_conditions} \ref{item:main_conditions_general_decay}---\ref{item:main_conditions_weak_not_strong} and finalize the proof of \Cref{conj:main_statement} in the two mentioned cases -- \Cref{section:Hilbert} covers the case where $X$ is a Hilbert space (\Cref{thm:main_hilbert}), while \Cref{section:proof_lp} considers $X = \ell^{p}$, $1\leq p < \infty$, and diagonal Gaussian measures (\Cref{thm:main_lp}).

\begin{remark}
Apart from  providing a ``skeleton'' for the proof of \Cref{conj:main_statement}, the strength of \Cref{thm:main_general} lies in its generality: It holds for any separable Banach space and thereby paves the way for future research.
Further, remarkably, while \Cref{cond:main_conditions} \ref{item:main_conditions_general_decay}---\ref{item:main_conditions_weak_not_strong} are stated in terms of the prior measure $\mu$, the conclusions are drawn for MAP estimators of $\mu^{y}$, with \Cref{ass:Phi} providing the sufficient conditions for comparability between prior and posterior in order to make this possible.
\end{remark}

\subsection{A framework for proving existence of MAP estimators}

While we use the proof strategy described above to prove \Cref{thm:main_hilbert,thm:main_lp}, it paves the way for further research.
Note that \Cref{thm:main_general} is applicable to any separable Banach space, so this approach can be followed to prove \Cref{conj:main_statement} for other classes of Banach spaces.

\begin{condition}
\label{cond:main_conditions}
Under \Cref{ass:Phi}, we introduce the following four conditions:

\begin{enumerate}[label = (C\arabic*)]
    \item
    \label{item:main_conditions_general_decay}
    \emph{(vanishing condition for unbounded sequences)} --
    For any null sequence $(\delta_{m})_{m \in \N}$ in $\R^{+}$ and unbounded sequence $(x_m)_{m\in\N}$ in $X$, \[\liminf_{m \to \infty} \Ratio{x_m}{0}{\delta_{m}}{\mu} = 0.\]
    
    \item
    \label{item:main_conditions_existence_weak_limit}
    \emph{(weakly convergent subsequence condition)} -- 
    If $(\delta_{m})_{m \in \N}$ is a null sequence in $\R^{+}$ and $(x_m)_{m\in\N}$ is a bounded sequence in $X$ such that there exists $K>0$ satisfying, for each $m\in\N$, $\Ratio{x_m}{0}{\delta_{m}}{\mu} \geq K$, then $(x_m)_{m\in\N}$ has a weakly convergent subsequence.

    \item 
    \label{item:main_conditions_outside_E}
    
    \emph{(vanishing condition for weak limits outside $E$)} --
    For any null sequence $(\delta_{m})_{m \in \N}$ in $\R^{+}$ and weakly convergent sequence $(x_{m})_{m \in \N}$ in $X$ with weak limit $\bar z \notin E$, $\liminf_{m \to \infty} \Ratio{x_m}{0}{\delta_{m}}{\mu} = 0$.
     \footnote{\label{footnote:Conditon_C3}This condition corresponds to \citep[Lemma 3.7]{dashti2013map} and \citep[Lemma 4.11]{kretschmann2019nonparametric}.
   While this 
   is sufficiently strong for our purposes, namely the proofs of the main \Cref{thm:main_hilbert,thm:main_lp}, we actually prove the stronger statement with $\limsup$ in place of $\liminf$ both for Hilbert spaces (\Cref{lem:newlemma37_hilbert}) as well as for $X = \ell^{p}$ (\Cref{lemma:C2_for_lp}).}

    \item
    \label{item:main_conditions_weak_not_strong}
    \emph{(vanishing condition for weakly, but not strongly convergent sequences)} -- 
    For any null sequence $(\delta_{m})_{m \in \N}$ in $\R^{+}$ and weakly, but not strongly convergent sequence $(x_{m})_{m \in \N}$ in $X$ with weak limit $\bar z \in E$, $\liminf_{m \to \infty} \Ratio{x_m}{0}{\delta_{m}}{\mu} = 0$.
    \footnote{This condition corresponds to \citep[Lemma 3.9]{dashti2013map} and \citep[Lemma 4.13]{kretschmann2019nonparametric}.}
\end{enumerate}
\end{condition}

\begin{theorem}
\label{thm:main_general}
Let \Cref{ass:Phi} hold and $(\zeta^{\delta})_{\delta >0}$ be any asymptotic maximizing family (AMF) in $X$.
Then there exist constants $K > 0$ and $\delta_0>0$, such that, for any $0 < \delta < \delta_0$,
\begin{equation}
\label{equ:main_hilbert_mu_ball_inequalities}
\Ratio{\zeta^\delta}{0}{\delta}{\mu}
\geq
K.
\end{equation}
It follows that:
\begin{enumerate}[label = (\alph*)]
    \item \label{item:main_general_convergence}
    If \cref{cond:main_conditions} \ref{item:main_conditions_general_decay} --\ref{item:main_conditions_weak_not_strong} hold, $(\zeta^{\delta})_{\delta >0}$ is an AMF in $X$ and $(\delta_{m})_{m\in\N}$ is a null sequence,
    then $(\zeta^{\delta_m})_{m\in\N}$ has a subsequence which converges strongly (in $X$) to an element $\bar w\in E$
    and any limit point $\bar{z}$ of $(\zeta^{\delta})_{\delta >0}$ lies in $E$ and is a MAP estimator for $\mu^{y}$. 
    
    \item
    \label{item:main_general_MAPOM}
    If \Cref{cond:main_conditions} \ref{item:main_conditions_outside_E} holds, then any MAP estimator for $\mu^{y}$ is an element of the Cameron--Martin space $E$, minimizes the OM functional and is a limit point of some AMF.
    
    \item
    \label{item:main_general_OMMAP}
    If \Cref{cond:main_conditions} \ref{item:main_conditions_outside_E} holds and $\mu^{y}$ has a MAP estimator $\bar{z}$, then any minimizer $\bar{x}\in E$ of the OM functional is also a MAP estimator.
\end{enumerate}
In particular, if \Cref{cond:main_conditions} \ref{item:main_conditions_general_decay} -- \ref{item:main_conditions_weak_not_strong} are satisfied, then \Cref{conj:main_statement} holds.
\end{theorem}

\begin{proof}
Due to \Cref{ass:Phi} \ref{item:third_condition_Phi} and \Cref{def:asymptoticmaximizer} there exists a family $(\eps^\delta)_{\delta > 0}$ such that $\eps^\delta \searrow 0$ for $\delta \searrow 0$, and, for any $0 < \delta \leq 1$,
\begin{equation}\label{eq:proofeq1}
    \begin{split}
        \mu^y(B_\delta(\zeta^\delta)) &> \frac{1-\eps^\delta}{Z}\cdot \sup_{z\in X} \int_{B_\delta(z)}e^{-\Phi(u)}\d\mu(u) \geq \frac{1-\eps^\delta}{Z}\cdot \int_{B_\delta(0)}e^{-\Phi(u)}\d\mu(u)\\
        &\geq  \frac{1-\eps^\delta}{Z}\cdot \int_{B_\delta(0)}e^{-L(1)}\d\mu(u)
        =
        \frac{1-\eps^\delta}{Z} e^{-L(1)}\mu(B_\delta(0)).
    \end{split}
\end{equation}
Furthermore, by \Cref{ass:Phi}\ref{item:first_condition_Phi}, for any $z\in X$ and $\delta > 0$,
\begin{equation}\label{eq:proofeq2}
    \begin{split}
        \mu^y(B_\delta(z)) &= \frac1Z \int_{B_\delta(z)}e^{-\Phi(u)}\d\mu(u) \leq  \frac{e^{-M}}{Z} \mu(B_\delta(z)).
    \end{split}
\end{equation}
Choosing $0 < \delta_{0} \leq 1$ such that $\eps^\delta < 1/2$ for each $0 < \delta < \delta_{0}$, and denoting $K \coloneqq e^{M-L(1)}/2$,
\begin{equation*}
\mu(B_\delta(\zeta^\delta)) 
\geq
Ze^M \mu^y(B_\delta(\zeta^\delta))
\geq
(1-\eps^\delta)\, e^{M-L(1)}\, \mu(B_\delta(0))
\geq
K \, \mu(B_\delta(0)),
\end{equation*}
proving \eqref{equ:main_hilbert_mu_ball_inequalities}.

\paragraph{Proving \ref{item:main_general_convergence}}
Consider the sequence $\zeta^{\delta_m}$ with $\delta_{m} \searrow 0$ as $m\to\infty$. Then
\begin{enumerate}[label = (\roman*)]
    \item
    \label{item:main_general_boundedness}
    \Cref{cond:main_conditions} \ref{item:main_conditions_general_decay} implies boundedness of $(\zeta^{\delta_m})_{m\in\N}$ in $X$,
     \item
    \label{item:main_general_limit}
    \Cref{cond:main_conditions} \ref{item:main_conditions_existence_weak_limit} implies that $(\zeta^{\delta_m})_{m\in\N}$ has a weakly (in $X$) convergent subsequence with weak limit point $\bar w \in X$.
    \item
    \label{item:main_general_limit_in_E}
    \Cref{cond:main_conditions} \ref{item:main_conditions_outside_E} implies that any weak (in $X$) limit point $\bar z \in X$ of $(\zeta^{\delta_{m}})_{m\in\N}$ lies in the Cameron--Martin space $E$.
    \item
    \label{item:main_general_weak_implies_strong}
    \Cref{cond:main_conditions} \ref{item:main_conditions_weak_not_strong} implies that any weak (in $X$) limit point $\bar z \in E$  of $(\zeta^{\delta_{m}})_{m\in\N}$ is also a strong (in $X$) limit point of $(\zeta^{\delta_{m}})_{m\in\N}$.
\end{enumerate}
In particular, there exists a subsequence of $(\zeta^{\delta_{m}})_{m\in\N}$ which converges strongly (in $X$) to some $\bar w\in E$.
This proves the first part of \ref{item:main_general_convergence}.

Now let $\bar z$ be any limit point of $(\zeta^{\delta})_{\delta >0}$ and $(\delta_{m})_{m\in\N}$ be such that $(\zeta^{\delta_{m}})_{m\in\N}$ converges (strongly) to $\bar z$.
Note that $\bar z\in E$ by \ref{item:main_general_limit_in_E}.
We set 
\[S := \max\{\|\bar z\|_{X}, \sup_{m\in\N} \|\zeta^{\delta_{m}}\|_{X}\}.\]
Using the local Lipschitz constant $L(r)$ for $\Phi$ on $B_r(0)$ (see \Cref{ass:Phi}\ref{item:third_condition_Phi}), we obtain, for any $m\in\N$,
\begin{align*}
    &\ratioy{\zeta^{\delta_m}}{\bar z}{\delta_m}
    =
    \exp(\Phi(\bar z)-\Phi(\zeta^{\delta_{m}}))\,  \frac{\int_{B_{\delta_{m}}(\zeta^{\delta_{m}})}e^{\Phi(\zeta^{\delta_{m}})-\Phi(u)}\, \d\mu(u)}
    {\int_{B_{\delta_{m}}(\bar z)}e^{\Phi(\bar z)-\Phi(u)}\, \d\mu(u)}
    \\
    &\quad\leq
    \exp\left(L(S)\cdot\|\zeta^{\delta_{m}} - \bar z\|_{X} + L(S+\delta_{m})\cdot \delta_{m} + L(S+\delta_{m})\cdot \delta_{m}\right)\ratio{\zeta^{\delta_m}}{\bar z}{\delta_m}
\end{align*}
Since $\zeta^{\delta_{m}} \to \bar z$ as $m\to\infty$, \Cref{lem:limsupballs} and \Cref{def:asymptoticmaximizer} of AMFs imply
\begin{align}
    &\limsup_{m\to\infty}\notag
    \, 
    \ratioy{\sup}{\bar z}{\delta_m}
    =
    \limsup_{m\to\infty}\, 
    \ratioy{\sup}{\zeta^{\delta_m}}{\delta_m}
    \ratioy{\zeta^{\delta_m}}{\bar z}{\delta_m}
    \\\notag
    &\quad\leq
    \limsup_{m\to\infty}\, 
    (1-\eps^{\delta_m})^{-1}
    \exp\left(L(S)\, \|\zeta^{\delta_{m}} - \bar z\|_{X} + 2L(S+\delta_{m})\cdot \delta_{m}\right) \ratio{\zeta^{\delta_m}}{\bar z}{\delta_m}
    \\
    &\quad\leq
    1.\label{equ:main_limsup_inequality}
\end{align}
If we can show that $\limsup_{\delta\searrow 0}\ratioy{\sup}{\bar z}{\delta}\leq 1$ 
(i.e.\ for any null sequence, not just for $(\delta_{m})_{m\in\N}$), then, since $\ratioy{\sup}{\bar z}{\delta} \geq 1 $
for each $\delta>0$, this implies that in fact $\lim_{\delta\searrow 0}\ratioy{\sup}{\bar z}{\delta} = 1$, 
proving that $\bar z$ is a MAP estimator and finalizing the proof.
For this purpose assume otherwise, i.e.\ there exists a null sequence $(\eps_{m})_{m\in\N}$ such that $\limsup_{m\to\infty} \, \ratioy{\sup}{\bar z}{\eps_m}
> 1$.

With the same argumentation as in \ref{item:main_general_boundedness}--\ref{item:main_general_weak_implies_strong}, there exists a subsequence of $(\zeta^{\eps_{m}})_{m\in\N}$, which, for simplicity, we denote by the same symbol, that converges strongly to some element $\bar{x} \in E$.
Similarly to \eqref{equ:main_limsup_inequality} we obtain
\begin{equation}
\label{equ:main_limsup_inequality_similar}
\limsup_{m\to\infty} \,\ratioy{\sup}{\bar x}{\eps_m} 
\leq 1.
\end{equation}
Now, since $\bar{x},\bar{z} \in E$, the property of the OM functional, \Cref{thm:OM}, guarantees the existence of the limit $\lim_{\delta\searrow 0} \ratioy{\bar x}{\bar z}{\delta}$ and therefore \eqref{equ:main_limsup_inequality} implies

\begin{equation}
\label{equ:limits_agree}
\lim_{m\to\infty}\ratioy{\bar x}{ \bar z}{\eps_m}
=
\lim_{\delta\searrow 0} \ratioy{\bar x}{\bar z}{\delta}
=
\lim_{m\to\infty} \ratioy{\bar x}{\bar z}{\delta_m}
\leq
\limsup_{m\to\infty} \ratioy{\sup}{\bar z}{\delta_m}
\leq 1.
\end{equation}
It follows from \eqref{equ:main_limsup_inequality_similar} and \eqref{equ:limits_agree} that
\[
1
<
\limsup_{m\to\infty} \, \ratioy{\sup}{\bar z}{\eps_m}
=
\limsup_{m\to\infty} \, \ratioy{\sup}{\bar x}{\eps_m}
\lim_{m\to\infty} \,\ratioy{\bar x}{\bar z}{\eps_m}
\leq
1,
\]
which is a contradiction, finalizing the proof.

\paragraph{Proving \ref{item:main_general_MAPOM}}
Now let $\bar z\in X$ be any MAP estimator (not necessarily the one obtained as the limit of $\zeta^{\delta_m}$).
Assuming $\bar z \notin E$ and considering the constant sequence $(\bar z)_{m\in\N}$ (clearly converging to $\bar{z}$), the vanishing condition for weak limits outside $E$, \Cref{cond:main_conditions} \ref{item:main_conditions_outside_E}, implies that 
\begin{align*}
    \liminf_{m\to\infty}\ratio{\bar z}{0}{\delta_m}=0
\end{align*}
for any null sequence $(\delta_{m})_{m\in\N}$.
Since the constant family $(\bar{z})_{\delta > 0}$ is an AMF for $\mu^{y}$ by \Cref{lem:MAP_AMF}, \eqref{equ:main_hilbert_mu_ball_inequalities} implies
\[
\liminf_{\delta\searrow 0}\ratio{\bar z}{0}{\delta}\geq K > 0.
\]

This contradiction proves $\bar{z}\in E$.
By definition of MAP estimators and \Cref{thm:OM}, it follows for any $z^\star \in E$ that
\begin{align*}
    1 
    =
    \lim_{\delta \searrow 0} \ratioy{\bar z}{\sup}{\delta}
    \leq  \lim_{\delta \searrow 0}\ratioy{\bar z}{z^\star}{\delta}= \exp(I(z^\star) - I(\bar z)).
\end{align*}
Hence, $I(z^\star) \geq I(\bar z)$ and $\bar{z}$ is a minimizer of the OM functional.
Finally, by \Cref{lem:MAP_AMF}, $\bar{z}$ is also a limit point of the constant AMF $(\bar{z})_{m\in\N}$.

\paragraph{Proving \ref{item:main_general_OMMAP}}
By \ref{item:main_general_MAPOM}, $\bar{z}\in E$ and minimizes the OM functional $I$, hence $I(\bar{z}) = I(\bar{x})$.
It follows from \Cref{thm:OM} that
\[
\lim_{\delta\searrow 0} \ratioy{\bar x}{\sup}{\delta}
=
\lim_{\delta\searrow 0} \ratioy{\bar x}{\bar z}{\delta}
\cdot
\lim_{\delta\searrow 0} \ratioy{\bar z}{\sup}{\delta}
=
\exp(I(\bar{z})-I(\bar{x})) \cdot 1
=
1,
\]
proving \ref{item:main_general_OMMAP}.

In summary, we have shown that each AMF (the existence of some AMF follows from \Cref{lem:MAP_AMF}) has a limit point $\bar{z}\in E$, which is a MAP estimator.
Furthermore, each limit point of an AMF lies in $E$ and is a MAP estimator.
In addition, any MAP estimator minimizes the OM functional and is a limit point of some AMF.
Finally, each minimizer of the OM functional is a MAP estimator.
Together, this proves \Cref{conj:main_statement}.
\end{proof}

\subsection{Some comments on the proof of \texorpdfstring{\Cref{cond:main_conditions} \ref{item:main_conditions_general_decay}---\ref{item:main_conditions_weak_not_strong}}{the conditions}}
The main obstacle in proving \Cref{thm:main_hilbert,thm:main_lp} is the verification of \Cref{cond:main_conditions} \ref{item:main_conditions_general_decay}---\ref{item:main_conditions_weak_not_strong}.
Let us shortly summarize one of the main ideas, demonstrated on the derivation of the vanishing condition for unbounded sequences \ref{item:main_conditions_general_decay} in the finite-dimensional setting $X = \R^{k}$, $k\in\N$:
Our aim is to show that, for any $\delta > 0$ the ratio $\ratio{x}{0}{\delta}$ decays to zero as $\norm{x}_{X} \to \infty$.
For this purpose we extract a certain prefactor from the integrals in the following way:
\begin{align*}
\ratio{x}{0}{\delta}
&=
\frac{\int_{B_{\delta}(x)} \exp\big(-\tfrac{1}{2}\absval{u}_{E}^{2}\big) \d u}
{\int_{B_{\delta}(0)} \exp\big(-\tfrac{1}{2}\absval{u}_{E}^{2}\big) \d u}
\\
&\leq
\frac{\sup_{v \in B_{\delta}(x)} \exp\big( -\tfrac{1}{2} L(v) \big)}
{\inf_{v \in B_{\delta}(0)} \exp\big( -\tfrac{1}{2} L(v) \big)}
\,
\frac{\int_{B_{\delta}(x)} \exp\big(-\tfrac{1}{2}(\absval{u}_{E}^{2} - L(u))\big) \d u}
{\int_{B_{\delta}(0)} \exp\big(-\tfrac{1}{2}(\absval{u}_{E}^{2} - L(u))\big) \d u}.
\end{align*}
If $L$ satisfies the following conditions,
\begin{enumerate}[label = (\roman*)]
\item
\label{item:L_condition_bounds}
there exists $\alpha > 0$ and $\kappa_{1},\kappa_{2} \geq 0$ such that, for each $v\in\R^{k}$,
\\
$\norm{v}_{X}^{\alpha} - \kappa_{1} \leq L(v) \leq \norm{v}_{X}^{\alpha} + \kappa_{2}$,
\item
\label{item:L_condition_nonnegative_convex}
$\absval{\quark}_{E}^{2} - L$ is non-negative and convex,
\end{enumerate}
then \ref{item:L_condition_nonnegative_convex} implies that, by Anderson's inequality, we can bound the remaining ratio of integrals from above by $1$, while \ref{item:L_condition_bounds} implies that, for any fixed $\delta >0$, the first fraction vanishes as $\norm{x}_{X} \to \infty$.

In separable Hilbert spaces $X = \cH$ a function $L$ satisfying \ref{item:L_condition_bounds}--\ref{item:L_condition_nonnegative_convex} is not hard to find (in both finite and infinite dimensions) since both $\norm{\quark}_{\cH}$ and $\absval{\quark}_{E}$ are quadratic.
In general separable Banach spaces the large discrepancy between the geometries induced by the norms $\norm{\quark}_{X}$ and $\absval{\quark}_{E}$ strongly complicates the search for such a function $L$, where convexity is particularly hard to ensure.
For $X = \ell^{p}$, the technical \Cref{prop:convexitybound_lp_all_p} guarantees the existence of such a function $L$.
This result together with \Cref{prop:bound_balls_lp_finite} can be seen as the crux to the results presented in this paper.

\section{The Hilbert space case\texorpdfstring{: Proof of \Cref{thm:main_hilbert}}{}}
\label{section:Hilbert}

In this section we treat the case where $X = \cH$ is a Hilbert space, i.e.\ we prove \Cref{thm:main_hilbert}.
These results have already been presented by \citet{dashti2013map}, with some corrections by \citet{kretschmann2019nonparametric}.
However, both of these manuscripts did not prove the existence of the central object in their proofs, namely the $\delta$-ball maximizing centers $z^\delta = \argmax_x \mu^y(B_\delta(x))$, which seems to be a highly nontrivial issue, see \citet{hefin}.
This section closes this theoretical gap by working with AMFs $\zeta^\delta$ defined by \Cref{def:asymptoticmaximizer} and serves two further purposes:

First, the Hilbert space case provides insight into the main ideas of the proof of \Cref{conj:main_statement} with fewer technicalities than the more general case $X = \ell^{p}$.
Second, we use a helpful statement from \citep{da2002second}, restated in \Cref{prop:normextraction} below, which simplifies the proofs considerably in comparison to \citep{dashti2013map,kretschmann2019nonparametric} and renders the proofs more streamlined.

\begin{notation}
\label{notation:Hilbert_setting}
Let $\cH$ be an infinite-dimensional separable Hilbert space and $\mu = \calN(0,Q)$ a centered and non-degenerate Gaussian measure on $\cH$.
As the covariance operator $Q$ of $\mu$ is a self-adjoint, positive, trace-class operator \citep{baker1973joint}, there exists an orthonormal eigenbasis $(e_k)_{k\in\N}$ of $Q$ in which $\mu = \otimes_{k\in\N} \calN(0,\sigma_k^2)$ is a product measure of one-dimensional Gaussian measures, where $Qe_k = \sigma_k^2 e_k$ and $\sigma_{k} > 0$ for each $k\in\N$ and $\sum_{k\in\N}\sigma_k^2 < \infty$.
We assume the eigenvalues to be decreasing, i.e.\ $\sigma_{1} \geq \sigma_{2} \geq \cdots$.
We write $D = \diag(d_1,d_2,\ldots) \coloneqq \sum_{k\in\N} d_{k}\, e_{k}\otimes e_{k}$ for any operator that is diagonal in the basis $(e_k)_{k\in\N}$.
Denoting $a_k \coloneqq \sigma_k^{-2}$ for $k \in \N$, the Cameron--Martin space of $\mu$ is given by
\begin{equation}
\label{equ:CM_Hilbert}
E = \{z\in \cH: ~|z|_E < \infty\},
\qquad
|z|_E^2 = \sum_{k=1}^\infty a_k \langle z,e_k\rangle_\cH^2,
\end{equation}
see \cite[Theorem 2.23]{da2014stochastic}.
Finally, we define the orthogonal projection operators $\Pi^k,\Pi_k\colon \cH \to \cH$, $k\in\N\cup \{ 0 \}$, by
\[
\Pi^k(x) := \sum_{j=1}^k \langle x, e_j\rangle_\cH \, e_j,
\qquad
\Pi_k(x) := x - \Pi^k(x).
\]
Note that $\Pi^{0} = 0$ and $\Pi_0 = \mathrm{Id}$.

\end{notation}

We start by reciting the following result which will allow us to ``extract an exponential rate'' by integrating over a slightly wider Gaussian measure:
\begin{proposition}[{\citealp[Proposition 1.3.11]{da2002second}}]\label{prop:normextraction}
If $\Gamma \colon \cH \to  \cH$ is self-adjoint and such that $Q^{1/2}\Gamma Q^{1/2}$ is trace class on $\cH$ and additionally $\langle x, Q^{1/2}\Gamma Q^{1/2}x \rangle_\cH < \|x\|_{\cH}^2$ for all $x\in\cH$.
Then for $\mu = \calN(0,Q)$ and $\nu = \calN(0, (Q^{-1}-\Gamma)^{-1})$ we have
\[
\frac{\d \mu}{\d\nu}(u) = \frac{\exp\left(-\frac 12 \langle \Gamma u , u \rangle_\cH\right)}{\sqrt{\det(I - Q^{1/2}\Gamma Q^{1/2})}}.
\]
\end{proposition}

\begin{remark}
In one dimension this boils down to the following: Let $\sigma >0$ and $\mu=\calN(0,\sigma^2)$.
Then, for any $\gamma < \sigma^{-2}$,
\begin{align*}
    \mu(A) 
    &= \frac{1}{\sqrt{2\pi \sigma^2}}\int_A \exp\left(-\gamma^2\frac{x^2}{2}\right) \exp\left( -\frac{x^2}{2\left(\frac{\sigma^2}{1-\gamma^2\sigma^2}\right)}\right)\d x\\
    &= \int_A \frac{\exp\left(-\gamma^2\frac{x^2}{2}\right)}{\sqrt{1-\gamma^2\sigma^2}}\d \nu(x),
\end{align*}
where $\nu = \calN(0,\frac{\sigma^2}{1-\gamma^2\sigma^2}) = \calN(0,(\sigma^{-2}-\gamma^2)^{-1})$.
\end{remark}

Then we can re-prove the following lemma (as already stated in \citep{dashti2013map} and \citep{kretschmann2019nonparametric}):

\begin{lemma}[{\citealp[Lemma 3.6]{dashti2013map}}]
\label{lem:bounds36_balls_dlsv}
Let \Cref{ass:Phi} hold and $X = \cH$ be a separable Hilbert space.
Then, using \Cref{notation:Hilbert_setting}, for any $\delta > 0$ and $z\in\cH$, and $n\in \N$,

\[
\ratio{z}{0}{\delta}\leq \exp\left( -\frac{a_n}{2}\left[(\|\Pi_{n-1}z\|_{\cH}-\delta)^2- \delta^2 \right]\right). 
\]
\end{lemma}

\begin{proof}
Using \Cref{notation:Hilbert_setting}, for arbitrary $n\geq n_0$, let $\Gamma = \diag(0,\ldots,0,r,\ldots,)$ with entries $0 < r < a_n$ starting at position $n$, such that $Q^{-1}-\Gamma = \diag(a_1,\ldots, a_{n-1},a_n-r, a_{n+1}-r,\ldots)$ is a valid precision (i.e. inverse covariance) operator of a Gaussian measure on $\cH$. This means that $\langle x, \Gamma x\rangle = r\|\Pi_{n-1} x\|_X^2$. This choice of $\Gamma$ fulfills the conditions of  \Cref{prop:normextraction}: First, $(Q^{-1}-\Gamma)^{-1}$ is a valid covariance operator:
\[
\sum_{i=n}^\infty (a_i-r)^{-1}
=
\sum_{i=n}^\infty \frac{a_{i}^{-1}}{1-r a_{i}^{-1}}
\leq
\frac{1}{1-r a_{n}^{-1}} \sum_{i=n}^\infty a_{i}^{-1}
=
\frac{a_{n}}{a_{n} - r} \sum_{i=n}^\infty \sigma_{i}^{2}
<
\infty.
\]
Second, since $Q$ is trace class \citep{baker1973joint}, so is
\[
Q^{1/2}\Gamma Q^{1/2} = \diag(0,\ldots,0,r\sigma_n^{2}, r\sigma_{n+1}^{2},\ldots).
\]
Finally, as $r < a_n = \sigma_n^{-2}$, and $\sigma_m^2 \leq \sigma_n^2$ for $m > n$, we also have that $r\sigma_m^2 \leq 1$ for all $m\geq n$, hence $\langle x,Q^{1/2}\Gamma Q^{1/2}  x\rangle \leq \|x\|_X^2$.

Thus, with $\nu = \calN(0, (Q^{-1}-\Gamma)^{-1})$, \Cref{prop:normextraction} implies for any $\delta > 0$:
\begin{align*}
\ratio{z}{0}{\delta}
&=
\frac{\int_{B_\delta(z)} e^{-\frac12\langle x, \Gamma x\rangle_{\cH}}\d\nu(x)}{\int_{B_\delta(0)} e^{-\frac12\langle x, \Gamma x\rangle_{\cH}}\d\nu(x)}
\\
&\leq
\frac{\exp(-\tfrac{r}{2}(\|\Pi_{{n-1}}z\|_{\cH}-\delta)^2)}{\exp(-\tfrac{r}{2}\delta^2)} \frac{\int_{B_\delta(z)} \d\nu(x)}{\int_{B_\delta(0)}\d\nu(x)}
\\
&\leq
\exp\left(-\frac{r}{2}\left[(\|\Pi_{n-1}z\|_{\cH}-\delta)^2-\delta^2\right]\right)
\end{align*}
due to Anderson's inequality (\Cref{thm:anderson2} with $\gamma = \nu$, $A=B_\delta(0)$ and $a=z$).
Since above inequality holds for any $0 < r < a_n$, it also holds for $r=a_{n}$ by continuity, and the claim follows.
\end{proof}

\begin{corollary}
Let \Cref{ass:Phi} hold and $X = \cH$ be a separable Hilbert space.
Then the vanishing condition for unbounded sequences, \Cref{cond:main_conditions} \ref{item:main_conditions_general_decay}, holds.
\end{corollary}
\begin{proof}

Let $(\delta_{m})_{m \in \N}$ be a null sequence in $\R^{+}$ and $(x_m)_{m\in\N}$ be an unbounded sequence in $X$.
We have to prove that for any $\eps > 0$ and any $m\in \N$ there exists a $m^\star \geq m$ such that 
\[
   \ratio{x_{m^\star}}{0}{\delta_{m^\star}}
    \leq
    \eps.
    \]
    Indeed, for arbitrary $\eps > 0$ and $m\in \N$ there exists $M > 0$ such that  $\frac{a_1 M^2}{4} \geq \log \eps^{-1}$.
    Since $(\delta_{m})$ is a null sequence, there exists $m_1 \geq m$ such that for all $n\geq m_1$, $\delta_n < M/4$.
    By unboundedness of $(x_m)_m$ we can find a $m^\star\geq m_1\geq m$ such that $\|x_{m^\star}\|_\cH \geq M$.
    Then, by \Cref{lem:bounds36_balls_dlsv},
    \begin{align*}
       \ratio{x_{m^\star}}{0}{\delta_{m^\star}} &\leq \exp\left( -\frac{a_1}{2}\left[(\|x_{m^\star}\|_{\cH}-{\delta_{m^\star}})^2- \delta_{m^\star}^2 \right]\right)\\
        &\leq \exp\left(-\frac{a_1}{2} \left[\frac{9M^2}{16} - \frac{M^2}{16} \right]\right) = \exp(-\frac{a_1M^2}{4}) \leq \eps\qedhere
    \end{align*}
    
\end{proof}

Similarly we can shorten the proof of the following lemma:

\begin{corollary}[{\citet[Lemma 3.7]{dashti2013map}, \citet[Lemma 4.11]{kretschmann2019nonparametric}}]
\label{lem:newlemma37_hilbert}
Let \Cref{ass:Phi} hold and $X = \cH$ be a separable Hilbert space.
Then the vanishing condition for weak limits outside $E$, \Cref{cond:main_conditions} \ref{item:main_conditions_outside_E}, is satisfied.

\end{corollary}

\begin{proof}
We use \Cref{notation:Hilbert_setting} throughout the proof.

Let $(\delta_{m})_{m \in \N}$ be a null sequence in $\R^{+}$ and $(x_m)_{m\in\N}$ be a weakly convergent sequence with weak limit $\bar{z} \notin E$.

\[
\ratio{x_{m}}{0}{\delta_{m}}
\leq
\eps.
\]
Let $\varepsilon > 0$.
Since $\bar{z}\notin E$, $\absval{\Pi^{n}\bar{z}}_{E} \to \infty$ as $n\to\infty$ by \eqref{equ:CM_Hilbert},

hence there exists $n \in \N$ (which we fix from now on) such that
\begin{equation}
\label{equ:technical_inequalities_in_newlemma37_hilbert}
\absval{\Pi^{n}\bar{z}}_{E} \geq 4,
\qquad
\exp(- \tfrac{3}{64} \absval{\Pi^{n}\bar{z}}_{E}^{2}) < \varepsilon.
\end{equation}
Note that $\Gamma \coloneqq \diag (a_1/2, a_2/2, \ldots, a_n/2, 0,0, \ldots)$ is a valid choice for the operator $\Gamma$ in \Cref{prop:normextraction} and observe that
\begin{equation}
\label{equ:technical_identity_in_newlemma37_hilbert}
\langle x,\Gamma x\rangle_{\cH} = \tfrac{1}{2} \absval{\Pi^n x}_{E}^{2},
\qquad
x\in\cH.
\end{equation}
Since weak convergence $x_m \rightharpoonup \bar z$ implies componentwise convergence, there exists $m_1\in \N$ such that, for any $m \geq m_1$, $\absval{\Pi^{n}(\bar{z} - x_m)}_{E} \leq 1$.
Since $(\delta_{m})_{m\in\N}$ is a null sequence, there exists $m^\star\geq m_1$ such that, for each $m\geq m^\star$, $\delta_{m}^2\leq \sigma_n^2/n$.
It follows from \eqref{equ:technical_inequalities_in_newlemma37_hilbert} for any $m\geq m^\star$, any $z \in B_{\delta_m}(x_m)$ and any $w \in B_{\delta_m}(0)$, denoting $x_{m,j},z_j,w_j$ for the $j$-th component of $x_m,z,w$, that
\begin{enumerate}[label = (\roman*)]
    \item
    $\absval{\Pi^{n}(x_m-z)}_{E}^{2}
    =
    \sum_{j=1}^{n} \sigma_{j}^{-2} \absval{x_{m,j} - z_{j}}^{2}
    \leq
    \sum_{j=1}^{n} \sigma_{n}^{-2} \delta_{m}^{2}
    \leq
    \sum_{j=1}^{n} n^{-1}
    =
    1$;
    \item
    $\absval{\Pi^{n} z}_{E}
    \geq
    \tfrac{1}{2}\absval{\Pi^{n} \bar{z}}_{E} + 
    \underbrace{\tfrac{1}{2} \absval{\Pi^{n} \bar{z}}_{E}}_{\geq 2} - \underbrace{\absval{\Pi^{n}(\bar{z} - x_m)}_{E}}_{\leq 1} - \underbrace{\absval{\Pi^{n}(x_m-z)}_{E}}_{\leq 1}
    \geq
    \tfrac{1}{2}\absval{\Pi^{n} \bar{z}}_{E}$;
    \item
    $\absval{\Pi^{n} w}_{E}^2
    =
    \sum_{j=1}^{n} \sigma_{j}^{-2} \absval{w_{j}}^{2}
    \leq
    \sum_{j=1}^{n} \sigma_{n}^{-2} \delta_{m}^{2}
    \leq
    \sum_{j=1}^{n} n^{-1}
    =
    1
    \leq
    \tfrac{1}{16}\absval{\Pi^{n} \bar{z}}_{E}^2$.
\end{enumerate}
Using \eqref{equ:technical_identity_in_newlemma37_hilbert} and Anderson's inequality (\Cref{thm:anderson2}) applied to the Gaussian measure $\nu$ on $\cH$ as defined in \Cref{prop:normextraction}, this implies, for any $m\geq m^\star$,
\begin{align*}
\ratio{x_m}{0}{\delta_m}
 &=
 \frac{\int_{B_{\delta_m}(x_m)} \exp\big(-\frac12\langle z, \Gamma z\rangle_{\cH}\big)\d\nu(z)}
 {\int_{B_{\delta_m}(0)} \exp\big(-\frac12\langle w, \Gamma w\rangle_{\cH}\big)\d\nu(w)}
 \\
 &\leq
 \exp\bigg(
 \tfrac{1}{4} \sup_{w \in B_{\delta_m}(0)} \absval{\Pi^{n} w}_{E}^{2}
 -
 \tfrac{1}{4} \inf_{z \in B_{\delta_m}(x_m)} \absval{\Pi^{n} z}_{E}^{2}
 \bigg)
 \frac{\nu(B_{\delta_m}(x_m))}{\nu(B_{\delta_m}(0))}
 \\
 &\leq
 \exp\big(
 \tfrac{1}{64} \absval{\Pi^{n} \bar{z}}_{E}^{2}
 -
 \tfrac{1}{16} \absval{\Pi^{n} \bar{z}}_{E}^{2}
 \big)
 \\
 &=
 \exp\big(
 -\tfrac{3}{64} \absval{\Pi^{n} \bar{z}}_{E}^{2}
 \big)
 <
 \varepsilon,
\end{align*}
proving the claim.
\end{proof}

\begin{corollary}[{\citealp[Lemma 3.9]{dashti2013map} and \citealp[Lemma 4.13]{kretschmann2019nonparametric}}]\label{lem:newlemma39_hilbert}
Let \Cref{ass:Phi} hold and $X = \cH$ be a separable Hilbert space.
Then the vanishing condition for weakly, but not strongly convergent sequences, \Cref{cond:main_conditions} \ref{item:main_conditions_weak_not_strong}, is satisfied.\qedhere
\end{corollary}

\begin{proof}

We use \Cref{notation:Hilbert_setting} throughout the proof.
Let $(\delta_{m})_{m \in \N}$ be a null sequence in $\R^{+}$ and $(x_m)_{m\in\N}$ converge weakly, but not strongly to $\bar{z} \in E$.
We will show that, for any $\eps > 0$ and $m_{1}\in\N$, there exists $m^\star \geq m_{1}$ such
\[
\ratio{x_{m^\star}}{0}{\delta_{m^\star}}
\leq
\eps.
\]
Now let $\eps > 0$ and $m_{1}\in\N$.
Since weak convergence $x_m\rightharpoonup \bar z$ implies $\|\bar z\|_\cH\leq \liminf_{m\to\infty} \|x_m\|_\cH$ and as the convergence is not strong by assumption, the Radon--Riesz property guarantees the existence of $c>0$ such that
\begin{equation}
\label{equ:Radon_Riesz_implication}
\limsup_{m\to\infty} \|x_m\| > \|\bar z\|_\cH + c.
\end{equation}

(Otherwise, $\lim_{m\to\infty} \|x_m\| = \|\bar z\|_\cH$, in which case weak convergence implies strong convergence.) Since $a_{k} \to \infty$ as $k\to\infty$, there exists $n \in \N$ (which we fix from now on) such that $a_{n} \geq - 24 c^{-2}\log\eps$. 


Since $(\delta_{m})_{m\in\N}$ is a null sequence and weak convergence $x_m \rightharpoonup \bar z$ implies componentwise convergence, \eqref{equ:Radon_Riesz_implication} guarantees the existenceof $m^{\star} \geq m_{1}$ such that $\delta_{m^{\star}}\leq c/6$, $\norm{\Pi^{n}(\bar{z} - x_{m^{\star}})}_{\cH} < c/2$ and $\|x_{m^{\star}}\|_{\cH} > \|\bar z\|_{\cH} + c$.
This implies
\begin{align*}
    \|\Pi_n x_{m}\|_\cH &\geq \|x_{m}\|_\cH - \|\Pi^n x_{m}\|_{\cH}
    > \|\bar z\|_\cH + c - \|\Pi^n(x_{m} - \bar z)\|_\cH-\|\bar z\|_\cH
    \geq c/2
\end{align*}
and \Cref{lem:bounds36_balls_dlsv} yields
\[
\ratio{x_{m^\star}}{0 }{\delta_{m^\star}}
  \leq
  \exp\left( -\frac{a_{n}}{2}\left[(\|\Pi_{n}x_{m^\star}\|_{\cH}-\delta_{m^\star})^2- \delta_{m^\star}^2 \right]\right)
  \leq
  \exp \left(-\frac{a_{n}c^2}{24}  \right)
  \leq
  \eps.\qedhere
\]

\end{proof}

\begin{proof}[Proof of \Cref{thm:main_hilbert}]
By \Cref{lem:bounds36_balls_dlsv,lem:newlemma37_hilbert,lem:newlemma39_hilbert},
\Cref{cond:main_conditions} \ref{item:main_conditions_general_decay}, \ref{item:main_conditions_outside_E} and \ref{item:main_conditions_weak_not_strong} are fulfilled, while the weakly convergent subsequence condition \ref{item:main_conditions_existence_weak_limit} follows from the reflexivity of $\cH$.
Hence, all statements follow from \Cref{thm:main_general}.

\end{proof}

\section{The case \texorpdfstring{$X = \ell^{p}$}{X = lp}: Proof of \texorpdfstring{\Cref{thm:main_lp}}{main result}}
\label{section:proof_lp}

In this section we will extend the results in \Cref{section:Hilbert} to the spaces $X = \ell^{p}$, $1\leq p < \infty$, i.e.\ we will prove \Cref{thm:main_lp}.
Note that \Cref{thm:main_lp} is an actual generalization of \Cref{thm:main_hilbert} since the covariance structure in a Hilbert space can always be ``diagonalized'' by choosing an orthonormal eigenbasis of the covariance operator, which is a consequence of the Karhunen--Loève expansion \citep[Theorem 2.21]{sprungk2017numerical}.
In other words, the Hilbert space case $(\cH, \mu)$ with an arbitrary non-degenerate Gaussian measure $\mu$ is equivalent to the case $(\ell^2, \otimes \mathcal N(0, \sigma_k^2))$, where $\sigma_k^2$ are the corresponding eigenvalues (note that the Cameron--Martin space $E$ respects this equivalence due to \eqref{equ:CM_Hilbert}), and the setting considered in this manuscript corresponds to the canonical generalization from $\ell^2$ to $\ell^p$, $1\leq p < \infty$.

While our proof strategy is quite similar to the one in \citep{dashti2013map}, the strong discrepancy between the geometries of the unit balls in $E$ and $X = \ell^{p}$ for $p \neq 2$ poses a strong obstacle when attempting to extract an exponential decay rate out of the ratio $\ratio{z}{0}{\delta}$ with fixed $\delta > 0$, similar to the statement of \Cref{lem:bounds36_balls_dlsv} in the Hilbert space case.

To see exactly why this is problematic, let us reiterate on the crucial line in the proof of \Cref{lem:bounds36_balls_dlsv}.
We set $n=1$ for simplicity, and we focus on the finite-dimensional case (or finite-dimensional approximation to the infinite-dimensional case) which allows to write the integrals with respect to Lebesgue measure.
Due to the fact that the Hilbert space norm coincides with an (unweighted) $\ell^2$-norm, we can extract a multiple of the Hilbert space norm out of the integral, where $\delta > 0$, $z\in\cH$ and $r > 0$ is a sufficiently small constant:
\begin{align*}
\ratio{z}{0}{\delta} &=\frac{\int_{B_\delta(z)} \exp\left(-\frac{1}{2}\left[a_1 x_1^2 + \cdots + a_N x_N^2\right]  \right) \d x}{\int_{B_\delta(0)}\exp\left(-\frac{1}{2}\left[a_1 x_1^2 + \cdots + a_N x_N^2\right]  \right) \d x} \\
&\leq
\frac{\sup_{x\in B_\delta(z)} \exp(-\tfrac{r}{2}\|x\|_{\cH}^2)}{\inf_{x\in B_\delta(0)} \exp(-\tfrac{r}{2}\|x\|_{\cH}^2))} \\
&\qquad \cdot \frac{\int_{B_\delta(z)} \exp\left(-\frac{1}{2}\left[(a_1-r) x_1^2 + \cdots + (a_N-r)x_N^2\right]  \right) \d x}{\int_{B_\delta(0)}\exp\left(-\frac{1}{2}\left[(a_1-r) x_1^2 + \cdots + (a_N-r)x_N^2\right]  \right)\d x} \\
&\leq
\frac{\exp(-\tfrac{r}{2}(\|z\|_{\cH}-\delta)^2)}{\exp(-\tfrac{r}{2}\delta^2)}
\end{align*}

where the second factor (the ratio of the remaining integrals) can be bounded by $1$ due to Anderson's inequality (\Cref{thm:anderson}) under some prerequisites: First, the ambient space norm $\|\quark\|_\cH$ needs to be dominated by (a multiple of) the Cameron--Martin norm such that the integrand is integrable --- this is also true for the Banach space case, simply by compact embedding of $E$ in $X$. Second, the function $|\quark|_E - r \|\quark\|_\cH$ needs to be convex. This is trivially the case in the Hilbert space case due to this difference being a positive definite quadratic, but does not generalize to the Banach space case. Indeed, $|\quark|_E - \beta \|\quark \|_p$ is not convex for $p = 1$ and any $\beta > 0$. This issue is solved (in the general $\ell^p$ case) by \Cref{prop:convexitybound_lp_all_p}, which demonstrates how to find functions $L$ such that $|\quark|_E^2 - \beta L(\quark)$ is convex and $L$ is a suitable surrogate of the ambient space norm $\| \quark \|_p$, see \Cref{fig:nonconvex} for an illustration.

\begin{figure}
    \centering
    \includegraphics[width=0.49\textwidth]{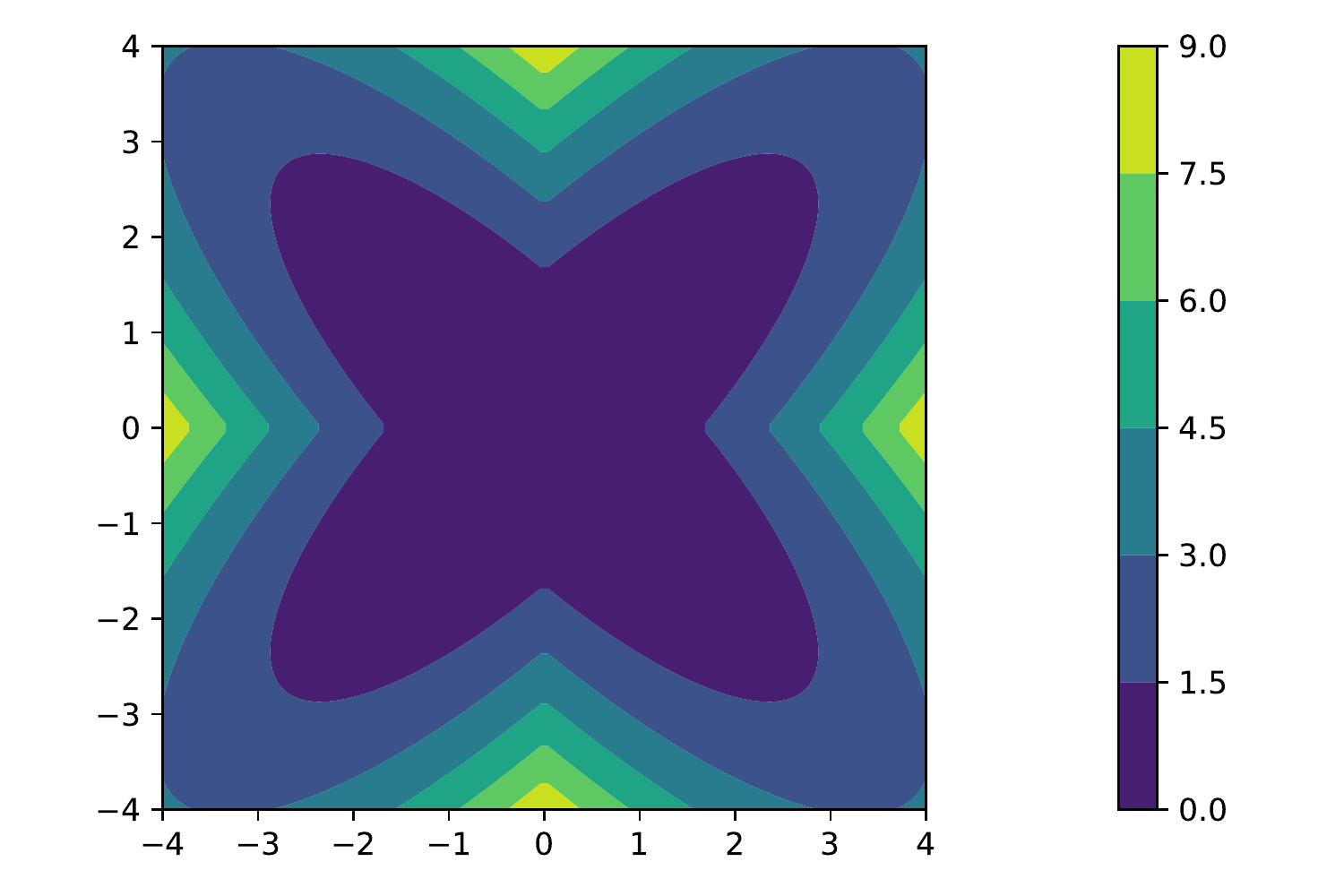}
    \includegraphics[width=0.49\textwidth]{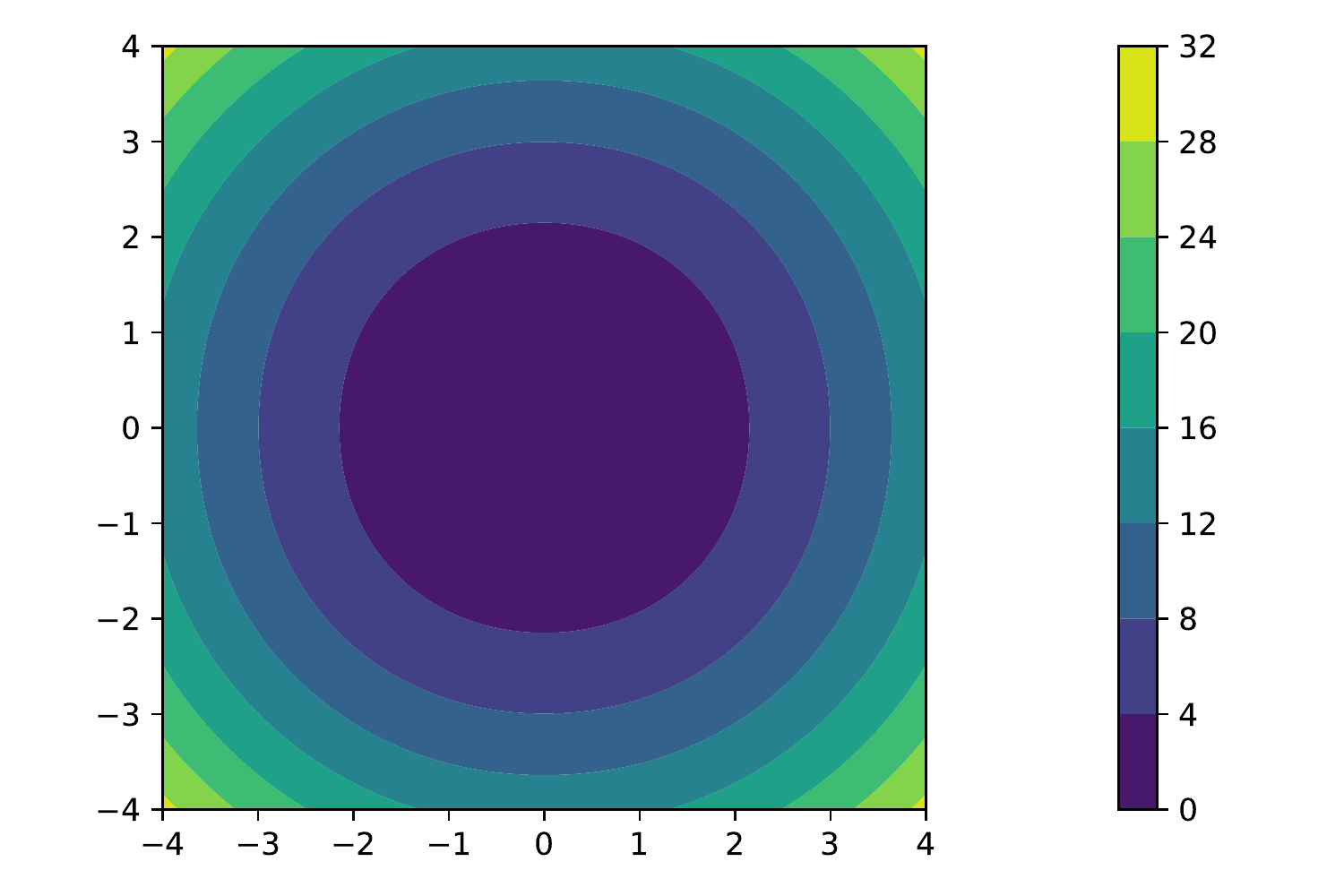}
    \caption{Visualization of the 2d case, $X = \ell^1$ and $\mu = \mathcal N(0,1)\otimes \mathcal N(0,1)$. \emph{Left:} Plot of the function $(x_{1},x_{2}) \mapsto x_1^2+x_2^2 - \beta (|x_1|+|x_2|)^2$ for a specific $\beta > 0$. The level sets show that this function is non-convex (this is indeed true for any $\beta > 0$). \emph{Right:} Plot of the function $(x_{1},x_{2}) \mapsto x_1^2+x_2^2 - \beta L(x_1,x_2)$ for suitable $\beta$, which is seen to be convex. 
    }
    \label{fig:nonconvex}
\end{figure}

\Cref{prop:bound_balls_lp_finite} then leverages this result towards a generalization of \Cref{lem:bounds36_balls_dlsv} in the $\ell^p$ case, after which the proof of validity of \Cref{cond:main_conditions} and subsequently \Cref{thm:main_lp} is more or less straight-forward.

When working in sequence spaces $X\subseteq \R^{\N}$, such as $\ell^{p}$ spaces, one important technique \citep{dashti2013map,ayanbayev2021gammab,agapiou2018sparsity} is to consider 
finite-dimensional approximations of $\mu(B_{\delta}(x))$, $x\in X$.
For this purpose, we introduce the following notation:

\begin{assumption}
\label{ass:general_assumption_lp}
We consider $X = \ell^{p} \coloneqq \ell^{p}(\N)$ with $1\leq p < \infty$ together with $\mu=\otimes_{j\in\N} \mathcal N(0,\sigma_{j}^2)$, a non-degenerate centred Gaussian measure on $X$ with diagonal covariance structure, where $\sigma_{1}\geq \sigma_{2} \geq \cdots > 0$ and $\sum_{j\in\N} \sigma_j^p <\infty$. 
\end{assumption}

\begin{remark}
The condition $\sum_{j\in\N} \sigma_j^p <\infty$ is a necessary condition for $\mu(X) = 1$ (i.e. samples $(x_i)_{i\in\N}$ are almost surely in $\ell^p$), see \citep[Lemma B.3]{ayanbayev2021gammab}.

\end{remark}

\begin{notation}
\label{notation:l_p_setting}
Let \Cref{ass:general_assumption_lp} hold.
Define
\[
\alpha \coloneqq \min(p,2),
\qquad
q \coloneqq \max ( p, 2(p-1)^{2}),
\qquad
S \coloneqq \Big( \sum_{j\in\N} \sigma_j^p \Big)^{1/p}.
\]
Further, for $k,K\in\N \cup \{0\}$ with $K>k$ define the projection operators $P^{k}\colon \R^{\N} \to \R^{k}$, $P_{k}\colon \R^{\N} \to \R^{\N}$, $P_{k}^{K}\colon \R^{\N} \to \R^{K-k}$ and $P^{-k}\colon \R^{k} \to \R^{\N}$ by
\begin{align*}
P^k(x) &\coloneqq (x_1,\ldots, x_k),&
P_k(x) &\coloneqq (x_{k+1}, x_{k+2},\ldots),
\\
P_{k}^{K}(x) &\coloneqq (x_{k+1}, \ldots, x_{K}),&
P^{-k}(u) &\coloneqq (u_1,\ldots, u_k,0,0,\ldots),
\end{align*}
where $P^{k} \coloneqq 0$ for $k=0$.
Accordingly, we define, for any $u\in\R^{k}$ and $v\in\R^{K}$,
\begin{itemize}
\item
$\displaystyle \absval{u}_{E^{k}} \coloneqq \sum_{j=1}^{k} \sigma_{j}^{-2} u_{j}^{2},
\qquad
\absval{v}_{E_{k}^{K}} \coloneqq \sum_{j=k+1}^{K} \sigma_{j}^{-2} v_{j}^{2}$,
\item
$\displaystyle B_{\delta}^{k}(u) \coloneqq \{ w\in\R^{k} \mid \norm{w-u}_{p} < \delta \}$,
\item 
$\displaystyle \mu_{k}=\otimes_{j=1}^{k} \mathcal N(0,\sigma_{j}^2)$.
\end{itemize}
Note that $\frac{1}{2}|\quark|_{E^k}$ is the negative log density of $\mu_{k}$.
\end{notation}

\begin{lemma}
\label{lemma:CM_lp}
If \Cref{ass:general_assumption_lp} holds, then the Cameron--Martin space of $(\ell^p, \mu)$ is given by $E = \{z\in\ell^p:~ |z|_E < \infty\}$ where $|z|_E^2 := \sum_{k=1}^\infty\frac{z_k^2}{\sigma_k^2}$.
\end{lemma}

\begin{proof}
By \citep[Lemma 3.2.2]{bogachev1998gaussian}, we may consider $\mu$ as a Gaussian measure on a Hilbert space $\cH \supseteq X$, into which $X$ is continuously and linearly embedded, without changing the Cameron--Martin space or its norm.
If $p\leq 2$, $X$ is continuously embedded in $\cH = \ell^{2} \supset X$, since $\norm{\quark}_{2} \leq \norm{\quark}_{p}$.
For $p > 2$, this can be accomplished by choosing any positive sequence $b \in \ell^{\frac{p}{p-2}}$ and $\cH \coloneqq \{x\in\R^\N \colon \|x\|_{\cH}^2 \coloneqq \sum_{k\in\N} b_k x_k^2 < \infty\}$, since, by Hölder's inequality,
\[
\norm{x}_{\cH}^{2}
=
\sum_{k\in\N} b_k x_k^2
\leq
\norm{b}_{\frac{p}{p-2}} \cdot \norm{ (x_{k}^{2})_{k\in\N} }_{\frac{p}{2}}
\leq
\norm{b}_{\frac{p}{p-2}} \cdot \norm{x}_{p}^{2}.
\]
The Cameron--Martin space and its norm for both $X$ and $\cH$ are therefore given by the well-known formulas \eqref{equ:CM_Hilbert}, see e.g.\ \citep[Theorem 2.23]{da2014stochastic}, proving the claim.
\end{proof}

In order to prove \Cref{thm:main_lp}, we will again proceed by showing \Cref{cond:main_conditions} \ref{item:main_conditions_general_decay} --- \ref{item:main_conditions_weak_not_strong} and then applying \Cref{thm:main_general}.
We start by showing the vanishing condition for weak limits outside $E$ \ref{item:main_conditions_outside_E}, while the vanishing condition for unbounded sequences \ref{item:main_conditions_general_decay} and the vanishing condition for weakly, but not strongly convergent sequences \ref{item:main_conditions_weak_not_strong} will require some additional work (\Cref{prop:convexitybound_lp_all_p,prop:bound_balls_lp_finite}).

\begin{lemma}
\label{lemma:C2_for_lp}
Under \Cref{ass:Phi,ass:general_assumption_lp}, for any family $(x^{\delta})_{0<\delta < 1}$ in $X$ and for any $\bar z \in X \setminus E$, such that $x^\delta \rightharpoonup \bar z$ converges weakly as $\delta \searrow 0$, we have
\[
\limsup_{\delta \searrow 0}\ratio{x^\delta}{0}{\delta}  = 0.
\]
In particular, the vanishing condition for weak limits outside $E$, \Cref{cond:main_conditions} \ref{item:main_conditions_outside_E}, is satisfied.
\end{lemma}

\begin{proof}
We use \Cref{notation:l_p_setting} throughout the proof.
Let $(x^{\delta})_{0<\delta < 1}$ be a family in $X$ and $\bar z \in X \setminus E$ such that $x^\delta \rightharpoonup \bar z$ converges weakly as $\delta \searrow 0$.
Let $0 < \eps < 1$ be arbitrary and $A \coloneqq \sqrt{8 \log(2/\eps)}$.
We proceed in four steps.
\vspace{2ex}

\textbf{Step 1:} There exist $K_{1}\in \mathbb N $ and $\delta_1 > 0$ such that, for each $u\in B_{\delta_{1}}^{K_{1}}({P^{K_1}\bar z})$, $\absval{u}_{E^{K_{1}}} \geq A$.
\vspace{1ex}

In order to see this, we assume the contrary, i.e.\ for each $K_{1} \in \N$ and $\delta_{1} > 0$, there exists $u\in B_{\delta_{1}}^{K_{1}}(P^{K_1}\bar z)$ with $\absval{u}_{E^{K_{1}}} < A$.
Then, for each $m\in\N$ (choosing $K_{1} = m$ and $\delta_{1} = m^{-1}$), there exists $u^{(m)} \in B_{m^{-1}}^{m}(P^m\bar z_{[1:m]})$ with $\absval{u^{(m)}}_{E^{m}} < A$.

Since $(P^{-m} u^{(m)})_{m\in\N}$ is bounded in $E$ by $A$, it has a weakly convergent (in $E$) subsequence, which, for simplicity, we also denote by $(P^{-m} u^{(m)})_{m\in\N}$, with weak limit $\bar u \in E$.
Further, since $u^{(m)} \in B_{m^{-1}}^{m}(P^m\bar z)$ for each $m\in\N$, $P^{-m}u^{(m)} \to \bar z$ strongly in $X$ as $m\to\infty$:
\[
\norm{P^{-m}u^{(m)} - \bar{z}}_{p}^{p}
=
\norm{u^{(m)} - \bar P^mz}_{p}^{p} + \norm{0 - P_{m} \bar{z}}_{p}^{p}
<
m^{-p} + \norm{P_{m} \bar{z}}_{p}^{p}
\xrightarrow[m\to\infty]{}
0.
\]
By considering each component $j\in\N$ separately, weak convergence in $E$ and (strong) convergence in $X$ imply
\[
u_{j}^{(m)} \xrightarrow[m\to\infty]{} \bar u_{j},
\qquad
u_{j}^{(m)} \xrightarrow[m\to\infty]{} \bar z_{j},
\qquad
j \in \N.
\]
Hence, by the uniqueness of the limit (in $\R$), we obtain the contradiction $E \ni \bar u = \bar z \notin E$.
\vspace{2ex}

\textbf{Step 2:} There exists $0 < \delta_2 < \delta_{1}/2$ such that, for each $0 < \delta < \delta_{2}$ and each $u\in B_{\delta}^{K_1}(P^{K_1}x^{\delta})$, we have that $\absval{u}_{E^{K_{1}}} \geq A$.
\vspace{1ex}

This can be seen as follows: Since $x^\delta \rightharpoonup \bar z$ converges weakly (and therefore componentwise) in $X$, there exists $0 < \delta_2 < \delta_{1}/2$ such that, for each $0 < \delta < \delta_{2}$, we have that $\norm{ P^{K_1}x^{\delta} - P^{K_1}\bar{z} }_{p} < \delta_{1}/2$.

Hence, for each $0 < \delta < \delta_{2}$ and each $u\in B_{\delta}^{K_1}(P^{K_1}x^{\delta})$,
\[
\norm{ u - P^{K_1}\bar{z} }_{p}
\leq
\norm{ u - P^{K_1}x^{\delta} }_{p}
+
\norm{ P^{K_1}x^{\delta} - P^{K_1}\bar{z} }_{p}
<
\delta + \tfrac{\delta_{1}}{2}
\leq 
\delta_{1},
\]
i.e.\ $B_{\delta}^{K_1}(P^{K_1}x^{\delta}) \subseteq B_{\delta_{1}}^{K_1}(P^{K_1}\bar z)$ for each $0 < \delta < \delta_{2}$, and the claim follows from Step 1.
\vspace{2ex}

\textbf{Step 3:}
There exists $0 < \delta^\star < \delta_2$ such that, for each $\delta < \delta^\star$ and each $u\in B_{\delta}^{K_1}(0)$, we have $|u|_{E^{K_1}} \leq A/\sqrt 2$.
\vspace{1ex}

This is evident from the fact that $|\quark|_{E^{K_1}}$ and $\|P^{K_1}\quark\|$ are equivalent norms on the (finite-dimensional) vector space $P^{K_1}X$.

\textbf{Step 4:}
For each $0 < \delta < \delta^\star$, $\ratio{x^\delta}{0}{\delta} \leq \eps$, finalizing the proof.
\vspace{1ex}

Let $0 < \delta < \delta^{\star}$.
For any $x\in X$, since $B_{\delta}(x) = \bigcap_{k\in\N} B_{\delta}^{k} (P^kx) \times \R^{\N \setminus \{ 1,\dots,k \} }$, the continuity of measures implies that $\mu_{k} (B_{\delta}^{k} (P^kx)) \to \mu(B_{\delta}(x))$.
Hence, there exists $k > K_{1}$

such that
\[
\Absval{ \ratio{x^\delta}{0}{\delta}
-
\frac{\mu_{k}(B_{\delta}^{k}(P^kx^\delta))}{\mu_{k}(B_{\delta}^{k}(P^k0))} }
\leq
\frac{\eps}{2}.
\]
Since, for any $x\in X$, $\R^k\ni v\in B_{\delta}^{k}(P^kx) $ implies $ P^{K_1}v\in B_{\delta}^{K_{1}}(P^{K_1}x) $, it follows from Steps 2 and 3 that
\begin{align*}
\ratio{x^\delta}{0}{\delta}
&\leq
\Absval{\ratio{x^\delta}{0}{\delta}
-
\frac{\mu_{k}(B_{\delta}^{k}(P^kx^{\delta}))}{\mu_{k}(B_{\delta}^{k}(P^k0))} }
+
\frac{\mu_{k}(B_{\delta}^{k}(P^kx^{\delta}))}{\mu_{k}(B_{\delta}^{k}(P^k0))}
\\
&\leq
\frac{\eps}{2} +
\frac{ \int_{B_{\delta}^{k}(P^kx^{\delta})} \exp\Big( -\frac{1}{2} \sum_{j=1}^{k} \tfrac{u_{j}^{2}}{\sigma_{j}^{2}} \Big) \, \d u }
{ \int_{B_{\delta}^{k}(0)} \exp\Big( -\frac{1}{2} \sum_{j=1}^{k} \tfrac{u_{j}^{2}}{\sigma_{j}^{2}} \Big) \, \d u }
\\
&\leq
\frac{\eps}{2} +
\frac{ \sup_{v\in B_{\delta}^{k}(P^kx^{\delta})} \exp\big( -\frac{1}{4} \absval{P^{K_1}v}_{E^{K_{1}}}^{2} \big) }
{\inf_{v\in B_{\delta}^{k}(0)} \exp\big( -\frac{1}{4} \absval{P^{K_1}v}_{E^{K_{1}}}^{2} \big)}
\\
&\qquad \cdot \frac{ \int_{B_{\delta}^{k}(P^kx^{\delta})} \exp\Big( -\frac{1}{4} \sum_{j=1}^{K_{1}} \tfrac{u_{j}^{2}}{\sigma_{j}^{2}} -\frac{1}{2} \sum_{j=K_{1}+1}^{k} \tfrac{u_{j}^{2}}{\sigma_{j}^{2}} \Big) \, \d u }
{ \int_{B_{\delta}^{k}(0)} \exp\Big( -\frac{1}{4} \sum_{j=1}^{K_{1}} \tfrac{u_{j}^{2}}{\sigma_{j}^{2}} - \frac{1}{2} \sum_{j=K_{1}+1}^{k} \tfrac{u_{j}^{2}}{\sigma_{j}^{2}} \Big) \, \d u }\\
&\leq
\frac{\eps}{2} +
\exp\big( - \tfrac{A^{2}}{4} + \tfrac{A^{2}}{8} \big) \cdot 1 =
\eps,
\end{align*}
where we bounded the last ratio of integrals by $1$ using Anderson's inequality (\Cref{thm:anderson}).

\end{proof}

As explained above, the following proposition implements a convexification of the function $|\quark|_E - \beta \|\quark \|_p$, which is necessary for the application of Anderson's inequality in the proof of \Cref{prop:bound_balls_lp_finite}:

\begin{proposition}
\label{prop:convexitybound_lp_all_p}
Using \Cref{notation:l_p_setting},
let $1\leq p < \infty$, let $k\in \N$ and $\rho\in\R^{k}$ with $\rho_{1} \geq \cdots \geq \rho_{k} > 0$.
Further, let $\gamma > 0$, let $\beta_{\ast} \coloneqq \frac{2\gamma^{2-\alpha}}{q\rho_{1}^{\alpha}}$ and let $0 \leq \beta < \beta_{\ast}$.
Then the functions
$L_{\rho,\gamma},\, f_{\rho,\beta,\gamma} \colon \R^{k} \to \R$
given by
\begin{align*}
L_{\rho,\gamma}(x)
&\coloneqq
\begin{cases}
\sum_{j=1}^{k} (\gamma^2\rho_{j}^2 + x_{j}^2)^{p/2} - (\gamma\rho_{j})^{p}
& \text{ if $1\leq p \leq 2$},
\\
\norm{x}_{p}^{2}
& \text{ if $2 < p < \infty$},
\end{cases}
\\
f_{\rho,\beta,\gamma}(x)
&=
\sum_{j=1}^{k} \tfrac{x_{j}^{2}}{\rho_{j}^{2}} - \beta L_{\rho,\gamma}(x),
\end{align*}
satisfy
\begin{enumerate}[label = (\alph*)]
    \item \label{item:L_proposition_inequalities}
    $\norm{x}_{p}^{\alpha} - \gamma^{\alpha} \norm{\rho}_{p}^{\alpha}
    \leq
    L_{\rho,\gamma}(x)
    \leq
    \norm{x}_{p}^{\alpha}$ for any $x\in\R^{k}$;
    \item \label{item:L_proposition_nonnegativity}
    $f_{\rho,\beta,\gamma}$ is non-negative;
    \item \label{item:L_proposition_convexity}
    $f_{\rho,\beta,\gamma}$ is convex.
\end{enumerate}
\end{proposition}

\begin{proof}
Recall that, for $0\leq p_{1} \leq p_{2} < \infty$, and $v \in \R^{n}$, $n\in\N$,
\begin{equation}
\label{equ:p_norm_inequality}
\norm{v}_{p_{1}} \geq \norm{v}_{p_{2}}.
\end{equation}
While \ref{item:L_proposition_inequalities} is trivial for $p>2$, it follows for $1\leq p \leq 2$ directly from the inequalities
$a^{q} \leq (a+b)^{q} \leq a^{q} + b^{q}$ for any $a,b \geq 0$ and $q\leq 1$, where the second inequality is a consequence of \eqref{equ:p_norm_inequality} for $v = (a,b)$:
\[
(a+b)^{q} = \norm{(a,b)}_{1}^{q}\leq \norm{(a,b)}_{q}^{q} = a^{q} + b^{q}.
\]

For \ref{item:L_proposition_nonnegativity}, note that, for any $\xi \in \R$, $1\leq p \leq 2$ and $r,\beta,\tau > 0$
\[
\frac{\xi^{2}}{r^{2}} - \beta (\tau^{2} + \xi^{2})^{p/2} + \beta \tau^{p}
\geq
0
~\Longleftrightarrow~
\beta \tau^{p} \left( 1 + \frac{1}{\beta \tau^{p}} \frac{\xi^{2}}{r^{2}}\right)
\geq
\beta \tau^{p}\left( 1 + \frac{\xi^{2}}{\tau^{2}}\right)^{p/2},
\]
which holds true, using Bernoulli's inequality with exponent $p/2 \leq 1$, for any $0 < \beta \leq \frac{2\tau^{2-p}}{pr^{2}}$:
\[
\left( 1 + \frac{\xi^{2}}{\tau^{2}}\right)^{p/2}
\leq
1 + \frac{p}{2}\frac{\xi^{2}}{\tau^{2}}
\leq
1 + \frac{\xi^{2}}{\beta \tau^{p} r^{2}}.
\]
By applying this observation componentwise with $r = \rho_{j}$ and $\tau = \gamma\rho_{j}$, we see that $f_{\rho,\beta,\gamma}$ is (globally) non-negative for any $0 < \beta \leq \min_{j=1,\dots,k} \frac{2\gamma^{2-p}}{p\rho_{j}^{p}} = \frac{2\gamma^{2-p}}{p\rho_{1}^{p}}$, proving \ref{item:L_proposition_nonnegativity} for any $1\leq p \leq 2$ (for $\beta = 0$ the claim holds trivially).
In the case $p>2$, \ref{item:L_proposition_nonnegativity} follows from \eqref{equ:p_norm_inequality}, since, for any $0 \leq \beta \leq \rho_{1}^{-2}$,
\[
\sum_{j=1}^{k} \tfrac{x_{j}^{2}}{\rho_{j}^{2}}
\geq
\rho_{1}^{-2} \norm{x}_{2}^{2}
\geq
\beta \norm{x}_{p}^{2}.
\]

For \ref{item:L_proposition_convexity}, first consider the case $1\leq p \leq 2$, for which the Hessian of $f_{\rho,\beta,\gamma}$ is diagonal.
Hence $f_{\rho,\beta,\gamma}$ is convex if and only if all those diagonal entries,
\[
\frac{\partial^{2} f_{\rho,\beta,\gamma}}{\partial x_{j}^{2}} (x)
=
\frac{2}{\rho_{j}^2} -
\beta p \, \frac{\gamma^{2}\rho_{j}^2 + (p-1) x_{j}^{2}}{(\gamma^{2}\rho_{j}^2 + x_{j}^2)^{2 - p/2}},
\qquad
j = 1,\dots,k,
\]
are non-negative functions.
Since, for $\tau > 0$, $\xi \in \R$ and $1\leq p \leq 2$,
\begin{equation}
\label{equ:convexitybound_lp_technical_inequality}
\frac{\tau^2 + (p-1) \xi^{2}}{(\tau^2 + \xi^2)^{2 - p/2}}
\leq
\frac{\tau^2 + \xi^{2}}{(\tau^2 + \xi^2)^{2 - p/2}}
=
\tau^{p-2}\, 
\frac{ 1 + \frac{\xi^{2}}{\tau^{2}} }{ \big(1 + \frac{\xi^{2}}{\tau^{2}}\big)^{2-p/2} }
\leq
\tau^{p-2}, 
\end{equation}
$f_{\rho,\beta,\gamma}$ is convex for each $0 \leq \beta < \min_{j=1,\dots,k} \frac{2 \gamma^{2-p}}{p \rho_{j}^{p}} = \frac{2\gamma^{2-p}}{p \rho_{1}^{p}}$ (by applying \eqref{equ:convexitybound_lp_technical_inequality} componentwise with $\tau = \gamma\rho_{j}$, $j=1,\dots,k$).

\vspace{1ex}

Now consider the case $2<p<\infty$.
The second-order partial derivatives of $L_{\rho,\gamma}$ for $x\neq 0$ are given by
\[
\frac{\partial^{2} L_{\rho,\gamma}}{\partial x_{l} \partial x_{m}} (x)
=
\begin{cases}
\frac{2 (p-1) |x_{l}|^{p-2}}{\norm{x}_{p}^{p-2}}
-
\frac{2 (p-2) |x_{l}|^{2p-2}}{\norm{x}_{p}^{2p-2}}
&
\text{if } l=m,
\\[2ex]
- \frac{2 (p-2)  \, x_{l} x_{m} \, |x_{l} x_{m}|^{p-2}}{\norm{x}_{p}^{2p-2}}
&
\text{if } l \not = m.
\end{cases}
\]
Hence, the Hessian of $f_{\rho,\beta,\gamma}$ for $x\neq 0$ can be written in the form
\[
\nabla^{2} f_{\rho,\beta,\gamma} (x)
=
\diag \big( ( 2\rho_{j}^{-2}
-
2 \beta (p-1) g_{j}(x) )_{j=1,\dots,k} \big)
+
2 \beta (p-2) h(x) h(x)^{\intercal},
\]
where $\diag(d_{1},\ldots,d_{k})$ denotes the $k\times k$ diagonal matrix with diagonal entries $d_{1},\ldots,d_{k}$ and the functions $g_{j} \colon \R^{k}\setminus \{0\} \to \R$, $j=1,\dots,k$, and $h\colon \R^{k}\setminus \{0\} \to \R^{k}$ are given by
\[
g_{j}(x) = \frac{|x_{j}|^{p-2}}{\norm{x}_{p}^{p-2}},
\qquad
h(x) = \bigg( \frac{ x_{j}\cdot |x_j|^{p-2}}{\norm{x}_{p}^{p-1}} \bigg)_{j=1,\dots,k}.
\]

Since $\absval{g_{j}} \leq 1$, $\nabla^{2} f_{\rho,\beta,\gamma}$ is symmetric and positive definite on the set $\R^k\setminus \{0\}$ for $0 \leq \beta < \frac{1}{(p-1)\rho_{1}^{2}}$.
In order to prove convexity, we show that for any $x,y\in \R^{k}$ and $\lambda \in [0,1]$,
\begin{equation}
\label{equ:convexity_def}
f_{\rho,\beta,\gamma}(\lambda x + (1-\lambda)y)
\leq
\lambda f_{\rho,\beta,\gamma}(x) + (1-\lambda)f_{\rho,\beta,\gamma}(y)    
\end{equation}
by considering the following three cases:
\paragraph{1. case: $x,y\neq 0$ and the line through $x$ and $y$ does not touch the origin $0\in \R^{k}$.} In this case, we can restrict the function $f_{\rho,\beta,\gamma}$ to an open half-space containing $x$ and $y$, but not containing $0\in\R^{k}$.
On this convex set, $f_{\rho,\beta,\gamma}$ is twice continuously differentiable and positive definiteness of the Hessian $\nabla^2 f_{\rho,\beta,\gamma}$ proves convexity, in particular \eqref{equ:convexity_def}.
\paragraph{2. case: $x,y\neq 0$ and the line through $x$ and $y$ contains the origin $0\in \R^{k}$.}
In this case, there exists $\lambda^\star \in (0,1)$ such that $\lambda^\star x + (1-\lambda^\star)y = 0$ and thereby $y = -\frac{\lambda^\star}{1-\lambda^\star}x$.
It follows for each $\lambda \in [0,1]$ that
\begin{align*}
    \lambda x + (1-\lambda) y &= (\lambda - \lambda^\star)x + ((1-\lambda)-(1-\lambda^\star))y + 0\\
    &= (\lambda - \lambda^\star)(x-y) = \frac{\lambda - \lambda^\star}{1-\lambda^\star}x.
\end{align*}
Since $f_{\rho,\beta,\gamma}(tx) = t^{2} f_{\rho,\beta,\gamma}(x)$ for each $t\in\R$,
\[
g(\lambda)
\coloneqq
f_{\rho,\beta,\gamma}(\lambda x + (1-\lambda) y)
=
f_{\rho,\beta,\gamma}\left(\frac{\lambda - \lambda^\star}{1-\lambda^\star}x\right)
=
\left(\frac{\lambda - \lambda^\star}{1-\lambda^\star}\right)^2 f_{\rho,\beta,\gamma}(x),
\]
which is a quadratic function in $\lambda$ with non-negative prefactor $f_{\rho,\beta,\gamma}(x) > 0$ (by \ref{item:L_proposition_nonnegativity}) and thereby convex.
Therefore, we obtain \eqref{equ:convexity_def} from
\begin{align*}
f_{\rho,\beta,\gamma}(\lambda x + (1-\lambda) y)
&=
g(\lambda \cdot 1 + (1-\lambda)\cdot 0)
\\
&\leq
\lambda g(1) + (1-\lambda) g(0)
\\
&=
\lambda f_{\rho,\beta,\gamma}(x) + (1-\lambda) f_{\rho,\beta,\gamma}(y).
\end{align*}
\paragraph{3. case: $x\neq 0$ and $y=0$}
In this case, \eqref{equ:convexity_def} follows from the previous cases by continuity:
\begin{align*}
f_{\rho,\beta,\gamma}(\lambda x + (1-\lambda) y)
&=
\lim_{t\searrow 0} f_{\rho,\beta,\gamma}(\lambda x + (1-\lambda) tx)
\\
&\leq
\lim_{t\searrow 0} \lambda f_{\rho,\beta,\gamma}(x) + (1-\lambda) f_{\rho,\beta,\gamma}(tx)
\\
&=
\lambda f_{\rho,\beta,\gamma}(x) + (1-\lambda) f_{\rho,\beta,\gamma}(y).\qedhere
\end{align*}
\end{proof}

\begin{remark}
Note that this bound on $\beta$ is not optimal. For example, for $n=2$, $p=4$ and $\rho_1=\rho_2=1$, we consider here $f_{\rho,\beta,\gamma}(x) = x^2+y^2 - \beta \sqrt{x^4+y^4}$. The lemma from above proves that this function is convex for $\beta < \frac{1}{3}$. In fact, it is convex already for $\beta < \sqrt2/3$ as can be shown by more elementary methods (exclusive to this low-dimensional setting). Note that in this specific case already $f_{\rho,\beta,\gamma}(x)\geq 0$ for $\beta \leq 1$.
\end{remark}

\begin{proposition}
\label{prop:bound_balls_lp_finite}

Under \Cref{ass:Phi,ass:general_assumption_lp} and using \Cref{notation:l_p_setting},
for each $0 < \delta < 1$, each $k \in \N \cup \{ 0 \}$, each $\gamma > 0$ and each $z\in X$,
\begin{equation}
    \ratio{z}{0}{\delta}
    \leq
    \exp\bigg( -\frac{\gamma^{2-\alpha}}{4 q \sigma_{k+1}^{\alpha}} \Big( (\norm{P_{k}z}_{p}-\delta)^{\alpha} - \gamma^{\alpha} S^{\alpha} - \delta^{\alpha} \Big) \bigg).
\end{equation}
\end{proposition}

\begin{proof}
Let $\beta \coloneqq \frac{\gamma^{2-\alpha}}{q \sigma_{k+1}^{\alpha}}$.
Let $K\in\N$ and $\rho = (\sigma_{k+1},\dots,\sigma_{K})$.

Observe that the function $f: \R^k\to \R$ defined by
\[
f(u)
=
\exp\big(  - \tfrac14 |u|_{E^{K}}^2 - \tfrac14 |u|_{E^{k}}^2 - \tfrac14 (|u|_{E_{k}^{K}}^2 - \beta L_{\rho,\gamma}(P_{k}^{K}u)) \big)
\]

is positive, symmetrical, integrable (since $f(u) \leq \exp (  - \tfrac14 |u|_{E^{K}}^2 )$ by \Cref{prop:convexitybound_lp_all_p} \ref{item:L_proposition_nonnegativity}) and log-concave (by \Cref{prop:convexitybound_lp_all_p} \ref{item:L_proposition_convexity}).
Hence, by \Cref{prop:convexitybound_lp_all_p} \ref{item:L_proposition_inequalities}, \ref{item:L_proposition_convexity} and Anderson's inequality (\Cref{thm:anderson}),
\begin{align*}
    \frac{\mu_{K}(B_\delta^{K}(z))}{\mu_{K}(B_\delta^{K}(0))}
    &=
    \frac{\int_{B_\delta^{K}(z)} \exp\big(-\tfrac12 |u|_{E^{K}}^2 \big) \d u}
    {\int_{B_\delta^{K}(0)} \exp\big(-\tfrac12 |u|_{E^{K}}^2\big) \d u}
    \\
    &\leq
    \frac{\sup_{v \in B_\delta^{K}(z)} \exp\big(-\tfrac{\beta}{4} L_{\rho,\gamma}(P_{k}^{K}v) \big)}
    {\inf_{v \in B_\delta^{K}(0)} \exp\big(-\tfrac{\beta}{4} L_{\rho,\gamma}(P_{k}^{K}v) \big)}
    \,
    \frac{\int_{B_\delta^{K}(z)} f(u) \d u}
    {\int_{B_\delta^{K}(0)} f(u) \d u}
    \\
    &\leq
    \exp\bigg( -\tfrac{\beta}{4} \Big( \inf_{v \in B_\delta^{K}(z)} ( \norm{P_{k}^{K}v}_{p}^{\alpha} - \gamma^{\alpha}\norm{\rho}_{p}^{\alpha}) - \sup_{v \in B_\delta^{K}(0)} \norm{P_{k}^{K}v}_{p}^{\alpha} \Big) \bigg)
    \\
    &\leq
    \exp\bigg( -\frac{\gamma^{2-\alpha}}{4 q \sigma_{k+1}^{\alpha}} \Big( (\norm{P_{k}^{K}z}_{p}-\delta)^{\alpha} - \gamma^{\alpha} S^{\alpha} - \delta^{\alpha} \Big) \bigg).
\end{align*}
For any $x\in X$, since $B_{\delta}(x) = \bigcap_{k\in\N} B_{\delta}^{k} (P^kx) \times \R^{\N \setminus \{ 1,\dots,k \} }$, the continuity of measures implies that $\mu_{k} (B_{\delta}^{k} (P^kx)) \to \mu(B_{\delta}(x))$.
Therefore, taking the limit $K\to\infty$ proves the claim.
\end{proof}

\begin{corollary}
\label{cor:C1_for_lp}
Under \Cref{ass:Phi,ass:general_assumption_lp} the vanishing condition for unbounded sequences, \Cref{cond:main_conditions} \ref{item:main_conditions_general_decay}, is satisfied.
\end{corollary}

\begin{proof}
We use \Cref{notation:l_p_setting} throughout the proof.
Let $(\delta_{m})_{m \in \N}$ be a null sequence in $\R^{+}$ and $(x_{m})_{m \in \N}$ be an unbounded sequence, i.e.\ there exists a subsequence $(x_{m_{n}})_{n \in \N}$ such that $\norm{x_{m_{n}}}_{p} \to \infty$ as $n\to \infty$.
Using \Cref{notation:l_p_setting} and \Cref{prop:bound_balls_lp_finite} with $\gamma=1$ and $k=0$ we obtain
\[
\ratio{x_{m_n}}{0}{\delta_{m_n}}
\leq
\exp\bigg( -\frac{1}{4 q \sigma_{1}^{\alpha}} \Big( (\norm{x_{m_{n}}}_{p}-\delta_{m_n})^{\alpha} - S^{\alpha} - \delta_{m_n}^{\alpha} \Big) \bigg)
\xrightarrow[n\to\infty]{}
0,
\]
proving the claim.
\end{proof}

\begin{corollary}
\label{cor:existence_weak_limit}
Under \Cref{ass:Phi,ass:general_assumption_lp} the weakly convergent subsequence condition, \Cref{cond:main_conditions} \ref{item:main_conditions_existence_weak_limit}, is satisfied.

\end{corollary}

\begin{proof}
We use \Cref{notation:l_p_setting} throughout the proof.
If $p>1$, the statement follows directly from the reflexivity of $X = \ell^{p}$.
Now let $p=1$, let $(\delta_{m})_{m \in \N}$ be a null sequence in $(0,1)$ and $(x_m)_{m\in\N}$ be a bounded sequence in $X$ satisfying, for some $K>0$ and each $m\in\N$, $\ratio{x_m}{0}{\delta_m}\geq K$.

We first show that $(x_{m})_{m\in \N}$ is equismall at infinity, i.e.\ for every $r>0$ there exists $k\in \N$ such that, for each $m \in \N$, $\|P_k x_{m}\|_1 < r$.
Assuming the contrary, there exists $r>0$ such that, for any $k\in\N$, there exist $m_k \in \N$ such that $\|P_k x_{m_k}\|_1 \geq r$.

If the sequence $(m_{k})_{k\in\N}$ was bounded by some $N\in\N$, then, using the fact that $\lim_{k\to\infty}\|P_k x\|_1 = 0$ for any (fixed) $x\in X$,
\[
r
\leq
\limsup_{k\to\infty } \|P_k x_{n_k}\|_1
\leq
\lim_{k\to\infty}\sup_{n=1,\ldots, N} \|P_k x_{n}\|_1
=
0
<
r.
\]
Since this is a contradiction, $(m_{k})_{k\in\N}$ is unbounded.
Using $\sigma_{k} \searrow 0$ and $\delta_{k} \searrow 0$ as $k\to\infty$, this implies the existence of $k\in\N$ such that $\delta_{m_{k}} \leq r/8$ and \[\exp\left(-\frac{r^2}{32\sigma_{k+1} \sum_{j\in\N}\sigma_j}\right) < K.\]
Using \Cref{prop:bound_balls_lp_finite} with $\gamma \coloneqq \frac{r}{4\sum_{j\in\N} \sigma_j}$ we obtain

\begin{align*}
\ratio{x_{m_k}}{0}{\delta_{m_k}}
&\leq
\exp \left( -\frac{\gamma}{4\sigma_{k+1}}
\left(\|P_k x_{m_k}\|_1-2{\delta_{m_k}}-\gamma \sum_{j\in\N} \sigma_j \right)
\right)
\\
&\leq
\exp\left( -\frac{r}{16\sigma_{k+1} \sum_{j\in\N} \sigma_j } \left(r-\tfrac{r}{4}-\tfrac{r}{4} \right) \right)
\\
&\leq
\exp\left( -\frac{r^2}{32 \sigma_{k+1} \sum_{j\in\N} \sigma_j}\right)
\\
&<
K,
\end{align*}
contradicting the assumption $\ratio{x_m}{0}{\delta_m}\geq K$ for each $m\in\N$.

Hence, $(x_{m})_{m\in\N}$ is equismall at infinity and, combined with its boundedness, this implies the existence of a weakly convergent subsequence of $(x_{m})_{m\in\N}$ by \citep[Theorem 44.2]{treves2016topological}.
\end{proof}

\begin{corollary}
\label{cor:C3_for_lp}
Under \Cref{ass:Phi,ass:general_assumption_lp} the vanishing condition for weakly, but not strongly convergent sequences, \Cref{cond:main_conditions} \ref{item:main_conditions_weak_not_strong}, is satisfied.

\end{corollary}
\begin{proof}
We use \Cref{notation:l_p_setting} throughout the proof.
Let $(\delta_{m})_{m \in \N}$ be a null sequence in $\R^{+}$ and $(x_{m})_{m \in \N}$ be a weakly, but not strongly convergent sequence in $X$ with weak limit $\bar z \in E$,
\vspace{2ex}

\textbf{Step 1:} There exists a $c>0$ and $k_0\in\N$ such that, for any $k\geq k_0$,
\[\limsup_{m \to \infty} \norm{P_{k} x_{m}}_{X} > c.\]

There exists $A>0$ such that $\limsup_{m \to \infty} \norm{x_{m}-\bar{z}}_{X} > A$ (otherwise the convergence would be strong).
Let $c \coloneqq \tfrac{A}{2}$.
Since $\bar{z} \in E$, we have $\absval{P_{k}\bar{z}}_{E} \to 0$ as $k\to\infty$ by \Cref{lemma:CM_lp} and therefore $\norm{P_{k}\bar{z}}_{X} \to 0$ as $k\to\infty$ by continuous embedding $E\subset X$ \citep[Proposition 2.4.6]{bogachev1998gaussian}.
Hence, there exists $k_{0} \in \N$ such that, for each $k\geq k_{0}$, $\norm{P_{k}\bar{z}}_{X} < c$.
Let $k\geq k_{0}$ and assume the contrapositive, i.e.\ $\limsup_{m \to \infty} \norm{P_{k}x_{m}}_{X} \leq c$.
But then, since weak convergence implies componentwise convergence,
\begin{align*}
2c
=
A
&<
\limsup_{m \to \infty} \norm{x_{m}-\bar{z}}_{X}
=
\limsup_{m \to \infty} \norm{P^{k}(x_{m}-\bar{z}) + P_{k}x_{m} - P_{k}\bar{z}}_{X}
\\
&\leq
\underbrace{\limsup_{m \to \infty} \norm{P^{k}(x_{m}-\bar{z})}_{X}}_{=0\text{ by weak conv.}} +
\underbrace{\limsup_{m \to \infty} \norm{P_{k}x_{m}}_{X}}_{\leq c \text{ by assumption}} +
\underbrace{\norm{P_{k}\bar{z}}_{X}}_{<c \text{ since } k \geq k_{0}}
\\
&<
2c,
\end{align*}
which is a contradiction, proving the claim.

\vspace{2ex}

\textbf{Step 2:} For each $0 < \eps < 1$, $\liminf_{m \to \infty} \ratio{x_m}{0}{\delta_m} < \varepsilon$.

\vspace{2ex}

Let $0 < \eps < 1$, $\delta_{0} \coloneqq \tfrac{c}{4}$, $\gamma \coloneqq \tfrac{c}{4S}$ and $k \geq k_{0}$ such that
\[
\sigma_{k+1}
<
\left( \frac{c^2}{4^{4-\alpha} S^{2-\alpha} q (-\log \eps) } \right)^{1/\alpha}.
\]
Let $m_{0} \in \N$.
Using Step 1, there exists $m\geq m_{0}$ such that $\delta_{m} < \delta_{0} = \tfrac{c}{4}$ and $\norm{P_{k} x_{m}}_{X} > c$.
Since $\frac{3^{\alpha}-2}{4^{\alpha}} \geq \frac{1}{4}$ for $1\leq \alpha \leq 2$, and by setting $\gamma = \tfrac{c}{4S}$, \Cref{prop:bound_balls_lp_finite} implies
\begin{align*}
\ratio{x_m}{0}{\delta_m}
&\leq
\exp\bigg( -\frac{\gamma^{2-\alpha}}{4 q \sigma_{k+1}^{\alpha}} \Big( (\norm{P_{k}x_{m}}_{p}-\delta_{m})^{\alpha} - \gamma^{\alpha} S^{\alpha} - \delta_{m}^{\alpha} \Big) \bigg)
\\
&\leq
\exp\bigg( -\frac{(\tfrac{c}{4S})^{2-\alpha}}{4 q \sigma_{k+1}^{\alpha}} \Big( \big(\tfrac{3c}{4}\big)^{\alpha} - \big(\tfrac{c}{4}\big)^{\alpha} - \big(\tfrac{c}{4}\big)^{\alpha} \Big) \bigg)
\\
&\leq
\exp\bigg( -\frac{c^{2}}{4^{4-\alpha} S^{2-\alpha} q \sigma_{k+1}^{\alpha}} \bigg)
\\
&<
\eps.\qedhere
\end{align*}
\end{proof}

\begin{proof}[Proof of \Cref{thm:main_lp}]
By \Cref{lemma:C2_for_lp,cor:C1_for_lp,cor:existence_weak_limit,cor:C3_for_lp}, \Cref{cond:main_conditions} \ref{item:main_conditions_general_decay} -- \ref{item:main_conditions_weak_not_strong} are fulfilled and all statements follow from \Cref{thm:main_general}.

\end{proof}

\section{Conclusion}
We proved the existence of MAP estimators in the context of a Bayesian inverse problem for parameters in a separable Banach space $X$, where $X$ is either a Hilbert space or $X=\ell^p$, $p\in [1,\infty)$, with a diagonal Gaussian prior.
The Hilbert space case had been proven before by \citep{dashti2013map,kretschmann2019nonparametric}, however, they did not show the existence of the central object in their proofs, namely the $\delta$-ball maximizers $z^{\delta} = \argmax_{z\in X} \mu^{y}(B_{\delta}(z))$.
We fixed this gap by working with an asymptotic maximizing family (AMF) $(\zeta^{\delta})_{\delta > 0} \subset X$ defined by \Cref{def:asymptoticmaximizer} and strongly simplified their proof by employing \citep[Proposition 1.3.11]{da2002second}, restated in \Cref{prop:normextraction}.
We decided to present this elegant and simple proof even though the Hilbert space case can be understood as a special case of $X = \ell^{p}$ for $p=2$.
The case $p\neq 2$, on the other hand, turned out to require novel techniques to prove the corresponding results. The crucial mathematical argument in this case relies on a convexification of the difference $|\quark|_E^2 - \beta \|\quark\|_X^2$ (\Cref{prop:convexitybound_lp_all_p}). This allows to extract a suitable ``rate of contraction'' such that the ratio $\ratio{z}{0}{\delta}$ can be bounded for any fixed $\delta > 0$ by a function decaying exponentially in $\norm{z}_{X}$ (\Cref{prop:bound_balls_lp_finite}).

We have also outlined a general proof strategy in \Cref{section:Main_results} how similar results (i.e.\ \Cref{conj:main_statement}) can be obtained for further separable Banach spaces.
For this purpose, we filtered out four crucial conditions, namely \Cref{cond:main_conditions} \ref{item:main_conditions_general_decay}---\ref{item:main_conditions_weak_not_strong}, which need to be proven in the Banach space of interest, and then the corresponding result follows almost immediately from \Cref{thm:main_general}.

Note that our results rely strongly on the characteristics of the $\ell^{p}$ norm and the diagonal structure of the covariance matrix of the Gaussian measure.
We suspect that the generalization to Gaussian measures on arbitrary separable Banach spaces requires deeper insight into the compatibility between the ambient space's geometry and the Cameron--Martin norm.
We hope that our \Cref{thm:main_general} paves the way for future research in this direction.
\appendix

\section{Gaussian measures in Banach spaces}

\label{section:Technical_details}

In notation, we will mainly follow \citep{bogachev1998gaussian}.
The continuous (or topological) dual space of $X$ is denoted by $X^\star$, while $X'$ denotes its algebraic dual.
In some cases, we will assume that $X$ is a Hilbert space, in which case we write $X = \cH$ for clarity. The object $\mu$ will always be a centred Gaussian measure on $X$ (or $\cH$).
We denote the Cameron--Martin space by $(E,\innerprod{\quark}{\quark}_{E})$, where we write the Cameron--Martin norm with single bars in order to differentiate it from the ambient space norm: $\absval{u}_{E} \coloneqq \sqrt{\innerprod{u}{u}_{E}}$.

It turns out that the extension of the covariance operator
\[
R_\mu \colon X^\star \to (X^{\star})',
\qquad
(R_{\mu} f)(g) \coloneqq \langle f, g \rangle_{L^{2}(\mu)}
\]
to the reproducing kernel Hilbert space (RKHS) $X_{\mu}^{\star} \coloneqq \overline{X^\star}^{L^2(X,\mu)}$ of $\mu$ satisfies $R_\mu (X_{\mu}^\star) = E$ \citep[Theorem 3.2.3]{bogachev1998gaussian}, where $E$ is viewed as a subspace of $(X^{\star})'$.
In addition, $R_\mu\colon (X_\mu^\star, \langle \quark,\quark \rangle_{L^{2}(\mu)}) \to (E,\langle \quark,\quark\rangle_E)$ is an isometric isomorphism \citep[page 60]{bogachev1998gaussian} and satisfies the \emph{reproducing property}
\begin{equation}
    \label{eq:reproprop}
    f(h) = \langle R_\mu f, h\rangle_E,
    \qquad f\in X_\mu^\star,
    \quad h\in E,
\end{equation}
which follows from the above and from treating $h = R_{\mu}g$ (for some $g\in X_{\mu}^{\star}$) as an element of $(X^{\star})'$:
\[
f(h)
=
f(R_{\mu}g)
=
(R_{\mu} g)(f)
=
\langle f, g\rangle_{L^{2}(\mu)}
=
\langle R_\mu f, R_{\mu}g \rangle_E
=
\langle R_\mu f, h\rangle_E.
\]

\begin{remark}
In the special case where the measure is defined on a Hilbert space $\cH$, the covariance operator $R_\mu$ takes the form of a self-adjoint, non-negative trace-class operator: $R_\mu = Q$ where
\[
Q: Q^{-1/2}(X) = X_\mu^\star \to E = Q^{1/2} X.
\]
In addition, the CM inner product and norm take the form
\begin{equation}
\label{equ:E_norm_H_norm_Hilbert}
\langle u, v\rangle_E
=
\langle Q^{-1/2}u, Q^{-1/2}v\rangle_{\cH},
\qquad
|u|_{E}
=
\norm{Q^{-1/2} u}_{\cH}.\qedhere
\end{equation}
\end{remark}

A result we are going to use in this context is the following technical lemma:

\begin{restatable}[]{lemma}{lemweakly}
\label{lem:limsupballs}
Let $X$ be a separable Banach space and $\mu$ a centred Gaussian measure on $X$, $\bar z\in E$ and $x^\delta \rightharpoonup \bar z$ weakly in $X$. Then
\[\limsup_{\delta\to 0} \ratio{x^\delta}{\bar z}{\delta}\leq 1.\qedhere\]
\end{restatable}

\begin{proof}
For any $\hat h\in X^\star$, the Cameron--Martin formula \citep[Corollary 2.4.3]{bogachev1998gaussian} implies
\begin{align}
\begin{split}
\label{equ:proof_limsupballs_numerator}
    \mu(B_\delta(x^\delta))
    &=
    \int_{B_\delta(x^\delta)}\d\mu
    =
    \int_{B_\delta(x^\delta-R_\mu \hat h)}
    \exp\big(-\tfrac12 |R_\mu\hat h|_E^2 - \hat h(u)\big) \, \d\mu(u)
    \\
    &\leq
    \mu({B_\delta(x^\delta-R_\mu \hat h)})\, \exp\big(-\tfrac12 |R_\mu\hat h|_E^2\big)
    \sup_{u\in B_\delta(x^\delta -R_\mu \hat h)} e^{ - \hat h(u)}
    \\
    &\leq
    \mu({B_\delta(0)})
    \exp\big(-\tfrac12 |R_\mu\hat h|_E^2 - \hat h(x^\delta-R_\mu \hat h)\big)\sup_{u\in B_\delta(0)} e^{ - \hat h(u)},
\end{split}
\end{align}
where we used Anderson's inequality (\Cref{thm:anderson2}) in the last step.
Since
\[
\int_{B_\delta(0)} \exp(-(R_\mu^{-1}\bar z)(u)) \, \d\mu(u)
=
\int_{B_\delta(0)} \exp((R_\mu^{-1}\bar z)(u)) \, \d\mu(u)
\]
due to symmetry of the set $B_\delta(0)$, another application of the Cameron--Martin theorem yields
\begin{align}
\begin{split}
\label{equ:proof_limsupballs_denominator}
    \mu(B_\delta(\bar z))
    &= 
    \exp \big(-\tfrac{1}{2}|\bar z|_E^2\big)
    \int\limits_{B_\delta(0)} \exp(-(R_\mu^{-1}\bar z)(u))\, \d\mu(u)
    \\
    &=
    \exp \big(-\tfrac{1}{2}|\bar z|_E^2\big)
    \int\limits_{B_\delta(0)} \frac{\exp((R_\mu^{-1}\bar z)(u)) + \exp(-(R_\mu^{-1}\bar z)(u))}{2} \, \d\mu(u)
    \\
    &\geq
    \exp \big(-\tfrac{1}{2}|\bar z|_E^2\big) \, \mu(B_\delta(0)),
\end{split}
\end{align}
where we used the inequality $a+a^{-1} \geq 2$ for any $a>0$ (alternatively, \eqref{equ:proof_limsupballs_denominator} can be proven via Jensen's inequality).
Since $x^\delta \to \bar z$ weakly in $X$, it follows from \eqref{equ:proof_limsupballs_numerator} and \eqref{equ:proof_limsupballs_denominator} that, for any $\hat h\in X^\star$,
\begin{align*}
    \limsup_{\delta\searrow 0} \ratio{x^\delta}{\bar z}{\delta}
    &\leq
    \limsup_{\delta\searrow 0} 
    \exp\big(\tfrac{1}{2}|\bar z|_E^2 - \tfrac12|R_\mu\hat h|_E^2 - \hat h(x^\delta-R_\mu \hat h)\big)\!
    \sup_{u\in B_\delta(0)}e^{-\hat h(u)}
    \\
    &\leq
    \exp\big(\tfrac{1}{2}|\bar z|_E^2 - \tfrac12|R_\mu\hat h|_E^2 - \hat h(\bar z-R_\mu \hat h)\big)
    \\
    &=
    \exp\big(\tfrac{1}{2}|\bar z|_E^2 - \tfrac12|R_\mu\hat h|_E^2 -\langle R_\mu\hat h, \bar z -R_\mu\hat h\rangle_E \big)
    \\
    &=
    \exp\big(\tfrac{1}{2}|\bar z|_E^2 - \tfrac12|R_\mu\hat h|_E^2 -\langle \bar z, \bar z-R_\mu\hat h\rangle_E + |R_\mu\hat h-\bar z|_E^2\big),
\end{align*}
where we used the reproducing property \eqref{eq:reproprop}.
Choosing a sequence $(\hat h_n)_{n\in\N}$ in $X^\star$ such that $R_\mu\hat h_n \to \bar z$ strongly in $E$ (this is possible by density of $X^\star$ in $R_\mu^{-1}E$), replacing $\hat h$ by $\hat h_n$ in the above inequality and taking the limit $n\to\infty$ proves the claim.
\end{proof}

\begin{theorem}[{Anderson's inequality, version 1; \citealt[Theorem 3.10.25]{bogachev2007measure}}]\label{thm:anderson}
Let $A$ be a bounded centrally symmetric convex set in $\R^n$, $n\in\N$ and let $f\colon \R^{n} \to \R$ be
\begin{itemize}
    \item non-negative and locally integrable,
    \item symmetrical, i.e.\ $f(-x) = f(x)$ for each $x\in\R^{n}$, and
    \item unimodal, i.e.\ the sets $\{f\geq c\}$ are convex for all $c>0$.
\end{itemize}
Then, for every $h\in \R^n$ and every $t\in [0,1]$, one has
\[\int_A f(x+th)\, \d x \geq \int_A f(x+h)\, \d x.\]
In particular, for every $z\in\R^n$, $\int_{z+A} f(x)\, \d x\leq \int_A f(x)\, \d x.$
\end{theorem}
\begin{theorem}[{Anderson's inequality, version 2; \citealt[Corollary 4.2.3]{bogachev1998gaussian}}]\label{thm:anderson2}
Let $\gamma$ be a centered Gaussian measure on a Banach space $X$. Let $A$ be a centrally symmetric convex set. Then for any $a\in X$, we have that $\gamma(A+a)\leq\gamma(A)$.
\end{theorem}

\section*{Acknowledgments}
The authors would like to express their gratitude to Birzhan Ayanbayev, Martin Burger, Nate Eldredge, Remo Kretschmann, Hefin Lambley, Han Cheng Lie, Claudia Schillings, Björn Sprungk, and Tim Sullivan for fruitful discussions and pointing out both errors and solution strategies.

\bibliographystyle{abbrvnat}
\bibliography{lit}

\end{document}